\documentclass[12pt]{amsart}

\usepackage{amsmath,amsthm,amssymb,latexsym,amscd}
\usepackage{setspace, bbding}
\usepackage{mathrsfs,amsfonts,graphicx,wasysym,textcomp}

\usepackage{graphicx,array}
\newcommand{\cA}{\mathcal{A}}
\newcommand{\cH}{\mathcal{H}}

\newcommand{\cD}{\mathcal{D}}
\newcommand{\cF}{\mathcal{F}}
\newcommand{\cE}{\mathcal{E}}
\newcommand{\cM}{\mathcal{M}}
\newcommand{\cO}{\mathcal{O}}
\newcommand{\cR}{\mathcal{R}}
\newcommand{\cS}{\mathcal{S}}
\newcommand{\fa}{\mathfrak{a}}
\newcommand{\fb}{\mathfrak{b}}
\newcommand{\fg}{\mathfrak{g}}
\newcommand{\fk}{\mathfrak{k}}
\newcommand{\fm}{\mathfrak{m}}
\newcommand{\fn}{\mathfrak{n}}
\newcommand{\fp}{\mathfrak{p}}
\newcommand{\fq}{\mathfrak{q}}
\newcommand{\fs}{\mathfrak{s}}
\newcommand{\fu}{\mathfrak{u}}
\newcommand{\fz}{\mathfrak{z}}
\newcommand{\SU}{\mathrm{SU}}
\newcommand{\Sp}{\mathrm{Sp}}

\newcommand{\cU}{\mathcal{U}}
\newcommand{\bs}{\mathbf{S}}
\newcommand{\bS}{\mathbf{S}}
\newcommand{\Ci}{C^\infty}
\newcommand{\Cic}{C^\infty_c}

\newcommand{\Hr}{ \mathcal{PW}_r }

\newcommand{\pwn}{q_N}
\newcommand{\Dr}{\DS_r}
\newcommand{\HrH}{ \mathcal{PW}^{\HS}_r }

\newcommand{\DrH}{\DS^{\HS}_r}

\newcommand{\PWR}{\widetilde{\PW}_{H,R}^{\Z_2}}
\newcommand{\PWRP}{\OS_{R+|r|} \left( \C \times \snc \right)}
\newcommand{\PWN}{\pi_{N}}
\newcommand{\SW}{\mathcal{S}}
\newcommand{\DRH}{\cD_{H,R}(\Xi)}

\newcommand{\PW}{\mathcal{PW}}
\newcommand{\Z}{\mathbb{Z}}
\newcommand{\R}{\mathbb{R}}
\newcommand{\Rn}{\mathbb{R}^n}

\newcommand{\C}{\mathbb{C}}
\newcommand{\Cn}{\mathbb{C}^n}
\newcommand{\N}{\mathbb{N}}
\newcommand{\nb}{\mathbb{\beta}}
\newcommand{\re}{\mathrm{Re}}
\newcommand{\im}{\mathrm{Im}}
\newcommand{\HS}{\mathcal{H}}
\newcommand{\DS}{\mathcal{D}}
\newcommand{\OS}{\mathcal{O}}

\newcommand{\sn}{S^{n-1}}
\newcommand{\snc}{S^{n-1}_{\mathbb{C}}}
\newcommand{\ftc}{\FT_{\R}^c}
\newcommand{\FT}{\mathcal{F}}

\newcommand{\FTrn}{\FT_{\Rn}}
\newcommand{\FTrnc}{\FT^c_{\Rn}}

\newcommand{\RT}{\mathcal{R}}
\newcommand{\Vol}{\mathrm{Vol}}
\newcommand{\id}{\mathrm{id}}
\newcommand{\FTc}{\mathcal{F}^c}
\renewcommand{\:}{\, : \,}
\newcommand{\pr}{\mathrm{pr}}

\newcommand{\SO}{\mathrm{SO}}
\newcommand{\SOn}{\mathrm{SO(n)}}

\newcommand{\GL}{\mathrm{GL}}

\newcommand{\Tr}{\mathrm{Tr}}

\newcommand{\ad}{\mathrm{ad}}

\newcommand{\rI}{\mathrm{I}}

\numberwithin{equation}{section}

\newtheorem{thm}{Theorem}[section]
\newtheorem{lem}[thm]{Lemma}
\newtheorem{coro}[thm]{Corollary}

\newenvironment{Proof}[1][Proof.]{\begin{trivlist} \item[\hskip \labelsep {\bfseries #1}]}{\end{trivlist}}
\theoremstyle{definition}		

\newtheorem{rem}[thm]{Remark}

\numberwithin{equation}{section}
\newcommand{\note}[2][\null]{%
 \marginpar{
   \renewcommand{\baselinestretch}{.5}
   \vspace{-1em}\hrule\vspace{3pt}%
   \footnotesize\raggedright\textsf{#2\ifx#1\null\else\\\hfill---
     {\em #1}\fi}\vspace{1.5em}
 }%
}
\def\sideremark#1{\ifvmode\leavevmode\fi\vadjust{\vbox to0pt{\vss
 \hbox to 0pt{\hskip\hsize\hskip1em
\vbox{\hsize2cm\small\raggedright\pretolerance10000
 \noindent #1\hfill}\hss}\vbox to8pt{\vfil}\vss}}}

\usepackage{graphicx,array}
\usepackage[all]{xy}
\newcommand{\K}{\mathbf{K}}
\renewcommand{\H}{\mathbf{H}}

\begin{document}
\title{Paley-Wiener Theorems with respect to the spectral parameter}
\author{Susanna Dann}
\address{Mathematics Department\\
Louisiana State University\\
Baton Rouge, Louisiana}
\email{sdann@math.lsu.edu}
\author{Gestur {\'O}lafsson}
\address{Mathematics Department\\
Louisiana State University\\
Baton Rouge, Louisiana}
\email{olafsson@math.lsu.edu}
\thanks{The research of S. Dann and G. \'Olafsson was supported by NSF grant
DMS-0801010}


\begin{abstract}
One of the important questions related to any integral transform on a manifold $\cM$ or on a homogeneous space $G/K$ is the description of the image of a given space of functions. If $\cM=G/K$, where $(G,K)$ is a Gelfand pair, then the harmonic analysis is closely related to the representations of $G$ and the direct integral decomposition of $L^2(M)$ into irreducible representations. We give a short overview of the Fourier transform on such spaces and then ask if one can describe the image of the space of smooth compactly supported functions in terms of the spectral parameter, i.e., the parameterization of the set of irreducible representations in the support of the Plancherel measure for $L^2(M)$. We then discuss the Euclidean motion group, semisimple symmetric spaces, and some limits of those spaces.
\end{abstract}

\maketitle
\section*{Introduction}
\noindent
The aim of this article is to discuss Paley-Wiener type theorems in diffe\-rent settings. Here we understand the term ``Paley-Wiener type theorems'' to mean the following problem: Given a manifold $\cM=G/H$, where $G$ is a Lie group and $H$ a closed subgroup of $G$, and given a Fourier type transform on $\cM$, characterize the image of a given function space on $\cM$.
More often than not, those are spaces of smooth compactly supported functions. Similar statements for square-integrable functions are called \textit{Plancherel theorems}. The ``classical'' Paley-Wiener theorem  identifies the space of smooth, respectively square-integrable, compactly supported functions on $\Rn$ with certain classes of holomorphic functions on $\C^n$ of exponential growth, where the exponent is determined by the size of the support. Similar statements are also true for distributions.

Normalize the Fourier transform on $\Rn$ by
\[\widehat{f}(\lambda )=\cF_{\Rn}(f)(\lambda )=\int f(x)e^{-2\pi ix\cdot \lambda }\, dx\, .\]
Then a function or a distribution $f$ is supported in a ball of radius $r>0$ centered at the origin if and only if $\widehat{f}$ extends to a holomorphic function on $\C^n$ such that for some constant $C>0$
$$ \left| \widehat{f}(\lambda) \right| \le
\left\{ \begin{array}{lll}
					C (1+|\lambda |^2)^{-N}e^{2\pi r|\im \lambda|} &\text{ for all } N\in \N\, &\text{ if  } f \text{ is smooth,} \\
          C (1+|\lambda |^2)^N e^{2\pi r | \im \lambda |} &\text{ for some } N\in \N &\text{ if } f \text{ is a distribution}.
\end{array} \right.$$
For $ f\in L^2(\Rn)$ along with the condition $|\widehat{f}(\lambda)| \le C e^{2\pi r|\im \lambda|} $ one has to assume that $\widehat{f}|_{\Rn}\in L^2(\Rn)$, see \cite{rudin87} p. 375. We shall often refer to the following Paley-Wiener space. Let $r>0$. Denote by $\PW_{r}(\Cn)$ the space of entire functions $H$ on $\Cn$ such that $z^m H(z)$ is of exponential type $\leq r$ for every $m \in \N^n$. The vector space $\PW_{r}(\Cn)$ is topologized by the family of seminorms: $$ \pwn(H) := \sup\limits_{z \in \Cn} (1+|z|^2)^N e^{-r |\im (z)|} |H(z)| $$ with $N \in \N$. In this case one can turn the above bijection statement into a topological statement. Let $\Dr(\Rn)$ denote the space of smooth functions supported in a ball of radius $r$ centered at the origin and equipped with the Schwartz topology. Then the following is true

\begin{thm}(Classical Paley-Wiener theorem)
The Fourier transform $\FTrn$ extends to a linear topological isomorphism of $\Dr(\Rn)$ onto $\PW_{2\pi r}(\Cn)$ for any $r>0$.
\end{thm}

Yet $\Rn$ can also be represented as a homogeneous space: $\Rn \simeq G/\SO(n)$ with the orientation preserving Euclidean motion group $G=\Rn \rtimes \SO(n)$. This realization comes with its own natural Fourier transform derived from the representation theory of $G$, see \cite{OlafssonSchlichtkrullRT08} and Section \ref{se-EmotGr}. One can again give a description of those spaces, and in fact we will give two such descriptions. The descriptions are given in terms of the parameter in the decomposition of $L^2(\Rn)$ into irreducible representations of $G$ as well as some homogeneity conditions.

More precisely, take the example $\cM=\Rn$. Consider functions on $\Rn$ as even functions on $\R\times\sn$ and take the Fourier transform in the first variable. Then the image of compactly supported smooth functions with an additional homogeneity condition are functions holomorphic in both variables - the radial and spherical directions - which are times any polynomial of exponential growth, and have a homogeneous power series expansion in the radial and spherical variables when the latter is restricted to be real. We will provide proofs for several topological Paley-Wiener type theorems in the Euclidean setting as well as discuss Paley-Wiener type theorems in the general setting of symmetric spaces of compact and noncompact type.

This article is organized as follows. In Section \ref{SectionBasicNotation} we introduce the basic notation. In section \ref{VecValPW} we prove a Paley-Wiener type theorem for Hilbert space valued functions. Next we recall the definition of a Gelfand pair $(G,K)$ and the Fourier transform on the associated commutative space $G/K$ in section \ref{SectionGelfandPairs}. A concrete example can be found in section \ref{se-EmotGr}, where $\Rn$ is considered is a homogeneous space. In section \ref{SectionEuclideanPW} we prove a topological analog of a Paley-Wiener type theorem due to Helgason, which is also stated at the beginning of that section. Then we prove that the Fourier transform extends to a bigger space, namely $\C\times\snc$, where $\snc$ stands for the complexified sphere. We also describe the image of the Schwartz space.

In sections \ref{SectionSSSofNT} and \ref{SectionSSSofCT} we discuss the case of Riemannian symmetric spaces of noncompact and compact type.  Recent results \cite{OlafssonWolf2009} on the inductive limit of symmetric spaces are reviewed in the last section \ref{SectionILofSS}.

\section{Basic Notation}\label{SectionBasicNotation}
\noindent
In this section we recall some standard notation that will be used in this article. We will use the notation from the introduction without further comments. We refer to \cite{helgason65} for proofs and further discussion.

Let $\Rn$ and $\Cn$ denote the usual n-dimensional real and complex Euclidean spaces respectively. For $z=(z_1, \dots, z_n) \in \Cn$, the norm $|z|$ of $z$ is defined by $|z|:=(|z_1|^2 + \cdots + |z_n|^2)^{1/2}$. The canonical inner-product of two vectors $x$ and $y$ on $\Rn$ or $\Cn$ is denoted by $x \cdot y$. The inner-product on $\Rn$ extends to a $\C$-bilinear form $(z, \xi) := \sum\limits_{i=1}^n z_i \xi_i $ on $\Cn \times \Cn$. Let $\N$ be the set of natural numbers including 0. For $j=1,\ldots ,n$, let $\partial_j=\partial /\partial z_j$. For any multi-index $m=(m_1, \ldots , m_n) \in \N^n$ and $z \in \Cn$, put $|m|:=m_1+ \cdots + m_n$, $z^m:=z_1^{m_1} \cdots z_n^{m_n}$, and $D^m:=\partial_1^{m_1} \cdots \partial_n^{m_n} $.

Let $\cM$ be a (smooth) manifold of dimension n. For an open subset $\Omega$ of $\cM$, let $\Ci (\Omega )$ and $\Cic (\Omega)$ denote the spaces of smooth complex valued functions on $\Omega$ and smooth complex valued functions with compact support on $\Omega$, respectively. For each compact subset $K$ of $\Omega$, define a seminorm $|\cdot|_{K, \alpha}$ on $\Ci (\Omega )$ by
$$ |f|_{K, \alpha}:= \max\limits_{p \in K} |D^{\alpha} f(p)|, $$
with $\alpha \in \N^n$. The vector space $\Ci (\Omega )$ equipped with the topology defined by these seminorms becomes a locally convex topological vector space and is denoted by $\cE (\Omega)$. For each compact subset $K$ of $\Omega$, let $\cD_K(\Omega)$ be the subspace of $\Ci (\Omega )$ consisting of functions $f$ with supp$(f) \subseteq K$. The topology on $\cD_K(\Omega)$ is the relative topology of $\Ci (\Omega )$. The Schwartz topology on $\Cic (\Omega)$ is the inductive limit topology of the subspaces $\cD_K(\Omega)$ with $K \subseteq \Omega$. The space $\Cic (\Omega)$ with the Schwartz topology is denoted by $\cD (\Omega )$. If $\cM$ is a Riemannian manifold and $x_0\in\cM$ is a fixed base point, then $\cD_r(\cM)$ stands for the subspace of $\cD (\cM)$ of functions supported in a closed ball of radius $r>0$ centered at $x_0$. Similar notation will be used for other function spaces. The space of smooth rapidly decreasing functions on $\Rn$, the \textit{Schwartz functions}, will be denoted by $\SW(\Rn)$. It is topologized by the seminorms
\[|f|_{N,\alpha} := \sup\limits_{x\in\Rn} (1+|x|^2)^N |D^{\alpha} f(x)|\, , \quad N \in \N \text{ and } \alpha \in  \N^n\, .\]
The Fourier transform is a topological isomorphism of $\SW (\Rn)$ onto itself with the inverse given by $\FTrn^{-1} (g)(x)=\FTrn (g)(-x)$. It extends to a unitary isomorphism of order four of the Hilbert space $L^2(\Rn)$ with itself.

Denote by $\sn$ the unit sphere in $\Rn$ and by $d\omega$ the surface measure on $\sn$. We will sometimes use the normalized measure $\mu_n$ which is given by $\sigma_n\mu_n=d\omega$, where $\sigma_n:=2 \pi^{n/2}/ \Gamma(n/2)$ for $n \geq 2$. For $p\in \R$ and $\omega\in\sn$ denote by $\xi (p,\omega )=\{x\in\Rn\: x\cdot \omega =p\}$ the hyperplane with the normal vector $\omega$ at signed distance $p$ from the origin. Denote by $\Xi$ the set of hyperplanes in $\Rn$. Then, as $\xi (r,\omega )=\xi (s,\sigma)$ if and only if $(r,\omega)=(\pm s,\pm \sigma)$, it follows that $\R\times\sn\ni (r,\omega)\mapsto \xi( r,\omega)\in \Xi$ is a double covering of $\Xi$. We identify functions on $\Xi$ with the corresponding even functions on $\R\times \sn$, i.e., $f(r,\omega )=f(-r,-\omega)$. The \textit{Radon transform} $\cR f$ of a function $f \in \Cic(\Rn) $ is defined by
\[\mathcal{R}f(\xi) := \int_{\xi} f(x) dm(x)\, , \]
where $dm$ is the Lebesgue measure on the hyperplane $\xi$. Then $\cR f\in \Cic (\Xi)$. Moreover, $\mathcal{R}$ is continuous from $L^1(\Rn)$ to $L^1(\R\times\sn)$ and its restriction from $\SW(\Rn)$ into $\SW(\R\times\sn)$ is
continuous \cite{hertle83}, where $\SW(\R\times\sn)$ is the space of smooth functions $\varphi$ on $\R\times\sn$ satisfying that for any $ k,m \in \N$ and for any differential operator $D_{\omega}$ on $\sn$
$$ \sup\limits_{(r,\omega)\in\R\times \sn} (1 + r^2)^k \left| \partial_r^m(D_{\omega} \varphi)(r,\omega) \right|<\infty.$$

The Radon transform is related to the Fourier transform by the \textit{Fourier-Slice Theorem}
\begin{equation}\label{eq-FST}
\widehat f(r\omega )= \cF_{\R}(\cR f)(r,\omega)\, ,
\end{equation}
where the Fourier transform is taken in the first variable.

Denote by $\cS_H (\Xi)$ the space of smooth functions $f: \R\times \sn\to \C$ such that
\begin{enumerate}
\item $f$ is even, i.e. $f(r,\omega )=f(-r,-\omega)$;
\item $\displaystyle{\eta_{k,m,D_{\omega}}(f):=\sup_{(r,\omega)\in\R\times \sn}(1+r^2)^k\left|\partial_r^m D_{\omega} f(r,\omega )\right|<\infty}$ for all $k,m\in\N$ and for any $D_{\omega}$ a differential operator on $\sn$;
\item For each $k\in \N$, the function $\omega \mapsto \int_{-\infty}^\infty f(r,\omega )r^k\, dr$ is a homogeneous polynomial of degree $k$.
\end{enumerate}

The family $\{\eta_{k,m,D}\}$ defines a topology on $\cS_H (\Xi)$.

\begin{thm}\label{th-RadSchwartz}
The Radon transform is a topological isomorphism $\cS (\Rn)\to \cS_H (\Xi)$.
\end{thm}

\begin{proof} By Theorem 2.4 in \cite{helgason99} it is a bijection and by Corollary 4.8 in \cite{hertle83} it is continuous with a continuous inverse.\end{proof}

Let $\cD_H (\Xi ):=C^\infty_c (\Xi)\cap \cS_H (\Xi)$ with the natural topology. For $R>0$, let $\cD_{H,R}(\Xi):=\{f\in \cD_H (\Xi)\: f(r,\omega )=0\text{ for } |r|>R\}$. The topology on $\cD_{H,R}(\Xi)$ is given by the seminorms
$$ |f|_{m,D_{\omega}}:=\sup_{(r,\omega)\in[-R,R] \times \sn} \left|\partial_r^m D_{\omega} f(r,\omega )\right|<\infty,$$
where $m$ is in $\N$ and $D_{\omega}$ is any differential operator on $\sn$. The topology on $\cD_H (\Xi )$ is the inductive limit topology of the subspaces $\cD_{H,R}(\Xi)$ with $0<R<\infty$.

\begin{thm}\label{th-1.2} The Radon transform is a topological isomorphism $\cD_R(\Rn)\simeq \cD_{H,R}(\Xi)$.
\end{thm}

\begin{proof} By Theorems 2.4 and 2.6 and Corollary 2.8 in \cite{helgason99} it is a bijection and by Corollary 4.8 in \cite{hertle83} it is continuous with a continuous inverse.\end{proof}

\section{The Paley-Wiener Theorem for vector valued functions on $\Rn$}\label{VecValPW}
\noindent
There are many Paley-Wiener theorems or sometimes also called Paley-Wiener-Schwartz theorems in the literature. They establish a relation between some class of holomorphic functions and harmonic analysis of compactly supported functions or distributions. The classical Paley-Wiener theorem characterizes the space of compactly supported smooth functions on $\Rn$ by means of the Fourier transform. \footnote{This case is often referred to as "the PW theorem" or "the classical PW theorem". However the original work of R. Paley and N. Wiener \cite{PaleyWiener1934} is devoted to the case of square-integrable functions. The case of distributions was first proved by L. Schwartz \cite{schwartz52} and this case is due to L. H{\"o}rmander \cite{hoermander63}, Th. 1.7.7, p.21. We comply with the old-established labeling.}
The long-known Paley-Wiener theorems were discussed in the introduction. Here we will prove an analog for Hilbert space valued functions. It reduces to the classical result by taking the Hilbert space to be one dimensional.

Let $ \HS $ denote a complex separable Hilbert space with a complete orthonormal set $\{e_i\}_{i\in J}$, where $J$ is a finite or a countably infinite index set.  The norm in $\HS$ is denoted by  $\| \cdot \|$, and the inner-product of two elements $u,v\in \HS$ is denoted by $(u, v)$.

Let $r > 0$. The space of $\HS$-valued functions $\varphi: \Rn \rightarrow \HS $ such that for every $ u \in \HS$ the complex valued function, $ x \mapsto( \varphi(x), u)$ belongs to $\Dr(\Rn)$, is denoted by $\DrH = \DrH(\Rn, \HS)$. We let the topology on $\DrH$ be given by the seminorms
\begin{equation}\label{def-SeminormDr}
 \nu_{N ,u}(\varphi) :=\max_{|\alpha |\le N} \sup_{x\in\Rn} \left| D^\alpha \left(\varphi(x), u\right) \right|\, ,
\end{equation}
with $\alpha \in\N^n$, $N\in \N$, and $u \in\HS$. The same topology is defined by the seminorms
\begin{equation}\label{def-seminormDr2}
\widetilde{\nu}_{N,u}(\varphi ):=\max_{k\le N} |(1+\Delta )^k(\varphi (x),u)|\, .
\end{equation}
Restricting $u$ to be one of the elements in the orthonormal basis $\{e_i\}$ gives the same topology.
We remark without proof, as it follows easily from the one dimensional case applied to each of the functions $x\mapsto (\varphi (x),e_i)$, that $\DrH$ is a Fr{\'e}chet space.

Denote by $\HrH = \HrH(\C^n, \HS)$ the space of weakly-holomorphic functions $F:\C^n\to \HS$, which satisfy that for every $ u\in \HS $ and $N\in\N$
\begin{equation}\label{def-SemiNormPW}
\rho_{N, u}(F) := \sup\limits_{z \in \C^n} (1 + |z|^2)^N e^{-r |Im(z)|} |(F(z), u)|< \infty.
\end{equation}
Let $\HrH$ be topologized by the seminorms $\rho_{N,u}$. Again, it is enough to use the countable family of seminorms $\{\rho_{N,e_j}\}_{N,j}$. Hence $\HrH$ is a Fr{\'e}chet space.

\begin{lem} The space $\HrH$ and its topology can be defined using the seminorms
 $$ \rho_{N}(F) := \sup\limits_{z \in \Cn} (1 + |z|^2)^N e^{-r |Im(z)|} \|F(z)\|\, ,$$
 with $ N \in \N $.
\end{lem}

\begin{proof} It is clear that if $\rho_{N}(F)<\infty$, then $\rho_{N,u}(F)\le \|u\|\rho_{N}(F)<\infty$. For the other direction,
let $E:=\{ (1 + |z|^2)^N e^{-r |Im(z)|} F(z) \: z\in\C^n\}$. From the assumption it follows that the set $E$ is a weakly bounded.  Moreover, $\HS$ being a Hilbert space is locally convex and thus by the Theorem 3.18 in \cite{rudin91} the set $E$ is bounded. Hence the seminorms $\rho_{N, u}$ can be replaced by the seminorms $\rho_N$.
\end{proof}

\begin{lem}\label{lemWI} Let $ \varphi \in \DrH$ and $z\in\C^n$, then $x\mapsto \varphi (x)e^{-iz\cdot x}$ is weakly integrable and
\begin{equation}\label{eq-FourAb}
\left|\int_{\Rn} (\varphi (x) , u)e^{-iz\cdot x}\, dx\right| \le \Vol (B_r(0))\|\varphi \|_\infty \|u \| e^{r|\im z|} \, .
\end{equation}
\end{lem}

\begin{proof} This follows from
\begin{equation}\label{eq-extimate}
|(\varphi (x),u)e^{-iz\cdot x}|\le \|\varphi\|_\infty \|u\|e^{r|\im z|} \chi_{\overline{B_r(0)}}\, .
\end{equation}
\end{proof}

We define the Fourier transform of $\varphi \in\DrH$ as the weak integral
\[\widehat{\varphi}(y)=\cF (\varphi)(y):= \int_{\Rn} \varphi (x ) e^{-2\pi ix\cdot y}\, dx\, .\]

\begin{thm}\label{secondTh} (Paley-Wiener type theorem for $\DrH$)
If $\varphi \in \DrH$, then $\cF {\varphi}$ extends to a weakly holomorphic function on $\C^n$ denoted by $\FTc (\varphi)$, and $\FTc (\varphi)\in \PW^{\HS}_{2\pi r}$. Furthermore, the Fourier transform $\FTc$ is a linear topological isomorphism of $\DrH$ onto $\PW^{\HS}_{2\pi r}$. The inverse of $\FTc$ is given by the conjugate weak Fourier transform on $\Rn$.
\end{thm}

\begin{proof}
Equation (\ref{eq-extimate}) clearly shows that the integral $\int (\varphi (x),u)e^{-2\pi iz\cdot x}\, dx$ converges uniformly on every compact subset of $\C^n$ and is therefore holomorphic as a function of $z$. Moreover by the Theorem 3.27 in \cite{rudin91} the integral $\int \varphi(x) \, e^{-2\pi iz\cdot x}\, dx$ converges to a vector in $\HS$.  Partial integration and (\ref{eq-extimate}) show that
\[\rho_{N,u}(\FTc (\varphi)) \le C \, \nu_{2N,u}(\varphi) <\infty,\]
for some constant $C$.
In particular,  $\FTc(\varphi)\in \PW^{\HS}_{2\pi r}$ and $\FTc : \DrH \to \PW^{\HS}_{2\pi r}$ is continuous.

To show surjectivity, let $F \in \PW^{\HS}_{2\pi r}$. Then the function $ z \mapsto (F(z), u)=: F_u(z) \in \PW_{2\pi r}(\C^n)$. Define $\varphi_u:= \cF_{\R^n}^{-1}(F_u)$, then by the classical Paley-Wiener Theorem $\varphi_u \in \Dr(\R^n)$. Let $k>n/2$, so that $x\mapsto (1+|x|^2)^{-k}$ is integrable. Then
\[\int |(F(x),u)|\, dx\le \left( \rho_{k}(F) \int (1+|x|^2)^{-k}\, dx\,\right)\,  \|u\|\, .\]
Hence the integral
\[\varphi(y) = \int_{\R^n}F(x)e^{2\pi ix\cdot y}\, dx=:(\FTc)^{-1}(F)(y)\]
exists and $\varphi \in\DrH$.

Finally, integrating by parts, we obtain
\[ \sup\limits_{y\in\Rn} |(1+\Delta )^N \varphi_u (y)| \le \left( \int (1+|x|^2)^{-k}\, dx\, \|u\|\right)\, \rho_{N+k}(F).\]
This shows that the the map $F\mapsto\varphi$ is continuous. The claim now follows as $(\FTc)^{-1}\circ \FTc=\id_{\DrH}$ and $\FTc\circ (\FTc)^{-1}=\id_{\PW^{\HS}_{2\pi r}}$.
\end{proof}

\begin{rem}
The above theorem can be proved without the use of the classical Paley-Wiener theorem. Hence as pointed out at the beginning of this section it generalizes the classical result.
\end{rem}

\section{Gelfand Pairs}\label{SectionGelfandPairs}
\noindent
We recall the definition of a Gelfand pair $(G,K)$ and the basic facts about the Fourier transform on the associated commutative space $G/K$. These facts are derived from the abstract Plancherel formula for the group $G$, instead of, as commonly done, from the theory of spherical functions. A more detailed discussion can be found in \cite{wolf2007}.

Let $G$ be a Lie group and  $K\subset G$ a compact subgroup. Denote by $\ell $ the left regular representation: $\ell (a)f(x)=f(a^{-1}x)$ and by $\rho $ the right regular representation: $\rho (a)f(x)=f(xa)$. We often identify functions on $G/K$ with right invariant functions on $G$. For $1\le p \le \infty$, let
\begin{eqnarray*}
L^p(G/K)^K&=&\{f\in L^p(G)\: (\forall k_1,k_2\in K)\,\, \ell (k_1)\rho (k_2)f=f\}\\
&=&\{f\in L^p(G/K)\: (\forall k\in K)\,\, \ell (k)f=f\}\, .
\end{eqnarray*}
If $f\in L^1(G)$ and $g\in L^p(G/K)^K$, then
\[f*g(x)=\int_G f(y)g(y^{-1}x)\, dy\]
is well defined, $f*g\in L^p(G/K)$ and $\|f*g\|_p\le \|f\|_1\|g\|_p$. For $f\in L^1(G/K)^K$ and $g\in L^p(G/K)^K$, $f*g$ is left $K$-invariant. It follows that $L^1(G/K)^K$ is a Banach algebra. The pair $(G,K)$ is called a \textit{Gelfand pair} if $L^1(G/K)^K$ is abelian. In this case we call $G/K$ a \textit{commutative space}. In case $G/K$ is a commutative space, there exists a set $\Lambda\subseteq \widehat{G}$, where $\widehat{G}$ is the unitary dual of $G$, such that
\begin{equation}\label{eq-DirectInt}
(\ell,L^2(G/K))\simeq \int^\oplus_{\Lambda }(\pi_\lambda ,\HS_\lambda )\, d\mu (\lambda )\, ,
\end{equation}
where each $\pi_{\lambda}$ is an irreducible unitary representation acting on the Hilbert space $\HS_{\lambda}$. The important fact for us is, that this is a multiplicity one decomposition and $\dim \HS_\lambda^K=1$ for almost all $\lambda$. Here, as usually $\HS_\lambda^K$ stands for the space of $K$-fixed vectors in $\HS_\lambda$.

For details in the following arguments we refer to \cite{wolf2007}, for the case of Riemannian symmetric spaces of noncompact type see \cite{OlafssonSchlichtkrullRT08}. Let $p : \Lambda \to \int^\oplus \HS^K_\lambda d\mu (\lambda )$ be a measurable section such that $\|p_\lambda \|=1$ for almost all $\lambda$. For each $\lambda$, $p_\lambda$ is unique up to a multiplication by $z\in \C$ with $|z|=1$.

Recall the operator valued Fourier transform. For $f\in L^1(G)$ and a unitary representation $\pi$ of $G$,
\[\pi (f):=\int_G f(x)\pi (x)\, dx\in \mathrm{B}(\HS_\pi),\]
where $\mathrm{B}(\HS_\pi)$ stands for the space of bounded operators on $\HS_\pi$. Furthermore, $\|\pi (f)\|\le \|f\|_1$. We also recall, that for Type I groups, there exists a measure, \textit{the Plancherel measure} on $\widehat{G}$, such that
\begin{enumerate}
\item If $f\in C^\infty_c (G)$, then $\pi (f)$ is a Hilbert-Schmidt operator and
\begin{equation}\label{eq-HSnorm}
\|f\|_2^2=\int_{\widehat{G}}\|\pi (f)\|^2_{\text{HS}}\, d\mu (\pi )\, .
\end{equation}
\item The operator valued Fourier transform extends to $L^2(G)$ such that (\ref{eq-HSnorm}) still holds.
\item For $f\in C^\infty_c(G)$, $\displaystyle{f(x)=\int_{\widehat{G}}\Tr (\pi (x^{-1})\pi (f))\, d\mu (\pi)}$ pointwise and in $L^2$-sense otherwise.
\end{enumerate}
The projection $\pr : \HS_\pi \to \HS_\pi^K$ is given by
\[\pr (v)=\int_K \pi (k)v\, dk\, .\]
If $f\in L^1(G/K)$, then for $k\in K$
\begin{eqnarray*}
\pi (f)v&=&\int_G f(x)\pi (x) v\, dx\\
&=& \int_G f(xk^{-1})\pi (x)v\, dx\\
&=&\int_G f(x)\pi (x)\pi (k)v\, dx\, .
\end{eqnarray*}
As this holds for all $k\in K$, integration over $K$ gives:

\begin{lem} Let $f\in L^1(G/K)$. Then $\pi (f)=\pi (f)\pr$.
\end{lem}

Thus $\pi (f)$ is a rank-one operator and it is reasonable to define the \textit{vector valued Fourier transform} by
\[\widehat{f}(\lambda  ):=\cF_{G/K} (f) (\lambda):= \pi_\lambda (f) (p_\lambda )\, ,\]
where $p_\lambda\in\HS_\lambda^K$ as above.

\begin{lem}
If $f\in L^1(G)$ and $g\in C^\infty_c (G/K)$, then
\[\cF (f*g)(\lambda )=\pi_\lambda (f)\widehat{g}(\lambda)\, .\]
\end{lem}

\begin{proof}
This follows from the fact that $\pi (f*g)=\pi (f)\pi (g)$.
\end{proof}

\begin{thm}\label{th-PlancherelGK} Let $f\in C_c^\infty (G/K)$. Then
\[\|f\|_2^2=\int \|\widehat{f}(\lambda )\|^2_{\HS_\lambda}\, d\mu (\lambda )\]
and
\[f(x)=\int (\widehat{f}(\lambda ),\pi_\lambda (x)p_\lambda )_{\HS_\lambda}\, d\mu (\lambda) .\]
Hence the vector valued Fourier transform extends to a unitary isomorphism
\[L^2(G/K)=\int^{\oplus} (\pi_\lambda ,\HS_\lambda )\, d\mu (\lambda )\]
with inverse
\[f(x)=\int (f_\lambda , \pi_\lambda (x)p_\lambda)_{\HS_\lambda }\, d\mu (\lambda)\]
understood in the $L^2$-sense.
\end{thm}

\begin{proof} Extend $e_{1,\lambda}:=p_\lambda$ to an orthonormal basis $\{e_{j,\lambda}\}_j$ of $\HS_\lambda$. As for $j>1$, $\pi_\lambda (f)e_{j,\lambda }=0$, we have
\begin{eqnarray*}
 \|\pi_\lambda (f)\|^2_{\text{HS}}&=&\Tr (\pi_\lambda (f)^*\pi_\lambda (f))\\
&=&(\pi_\lambda (f)^*\pi_\lambda (f)p_\lambda ,p_\lambda)_{\HS_\lambda}\\
&=&(\pi_\lambda (f)p_\lambda ,\pi_\lambda (f)p_\lambda)_{\HS_\lambda}\\
&=&|\widehat{f}(\lambda )|^2\, .
\end{eqnarray*}
Similarly,
\[\Tr (\pi_\lambda (x^{-1})\pi_\lambda (f))=(\pi_\lambda (x^{-1})\pi_\lambda (f)p_\lambda ,p_\lambda )_{\HS_\lambda}\, .\]
Hence, by the inversion formula for the operator valued Fourier transform
\[f(x)=\int (\widehat{f}(\lambda ), \pi_\lambda (x)p_\lambda )_{\HS_\lambda}\, d\mu (\lambda )\]
as claimed.

Given a section $(f_\lambda )\in \int^\oplus (\pi_\lambda, \HS_\lambda )\, d\mu$, define a rank-one operator section $(T_\lambda)$ by
\[T_\lambda p_\lambda = f_\lambda \text{ and } T_\lambda|_{(\HS_\lambda^K)^\perp} = 0\, .\]
Then $T_\lambda$ is a Hilbert-Schmidt operator and hence corresponds to a unique $L^2$-function
\begin{equation}\label{eq-InversionF}
f(x)=\int \Tr (\pi_\lambda (x^{-1})T_\lambda )\, d\mu (\lambda )=
 \int (f_\lambda , \pi_\lambda (x)p_\lambda)_{\HS_\lambda }\, d\mu (\lambda)
\end{equation}
and
\[\|f\|_2^2=\int \|f_\lambda \|^2\, d\mu (\lambda )\, .\]
As $x\mapsto \pi_\lambda (x)p_\lambda$ is right $K$-invariant, it follows that $f\in L^2(G/K)$. Furthermore, the abstract Plancherel formula gives that $\pi_\lambda (f)=T_\lambda$ and hence $\widehat{f}(\lambda )=f_\lambda$.
\end{proof}

Assume now that $f$ is left and right $K$-invariant. Then $\pi_\lambda (f)p_\lambda$ is $K$-invariant and hence a mutliple of $p_\lambda$, $\pi_\lambda (f)p_\lambda=(\widehat{f}(\lambda ),p_\lambda )_{\HS_\lambda }p_\lambda$. We have
\begin{eqnarray*}
(\widehat{f}(\lambda ),p_\lambda )_{\HS_\lambda }&=&\int_G f(x)(\pi_\lambda (x)p_\lambda ,p_\lambda )_{\HS_\lambda}\, dx\\
&=& \int_{G/K} f(x)\varphi_\lambda (x)\, dx,
\end{eqnarray*}
where $\varphi_\lambda (x)=(\pi_\lambda (x)p_\lambda ,p_\lambda )_{\HS_\lambda}$ is the \textit{spherical function} associated to $(\pi_\lambda,\HS_\lambda)$. Note that $\varphi_\lambda$ is independent of the choice of $p_\lambda$. Thus, in the $K$ bi-invariant case the vector valued Fourier transform reduces to the usual spherical Fourier transform on the commutative space $G/K$.

The question now is: How well does the vector valued Fourier transform on the commutative space $X$ describe the image of a given function space on $X$? Examples show that most likely there is no universal answer to this question. There is no answer so far for the Gelfand pair $(\mathrm{U}(n)\ltimes H_n, \mathrm{U}(n))$, where $H_n$ is the $2n+1$-dimensional Heisenberg group. Even though some attempts have been made to address this Paley-Wiener theorem for the Heisenberg group, \cite{fuhr2010,LudwigMolitor2009,NarayananThangavelu2006,LipsmanRosenberg96}. The Fourier analysis for symmetric spaces of noncompact type is well understood by the work of Helgason and Gangolli, \cite{gangolli71,helgason66}. On the other hand,
for compact symmetric spaces $U/K$, the Paley-Wiener theorem is only known for $K$-finite functions \cite{OlafssonSchlichtkrullPW08,OlafssonSchlichtkrull2010}.

In the following, we begin with one of the simplest cases of Gelfand pairs, the Euclidean motion group and $\SO(n)$. Then we discuss the case of Riemannian symmetric spaces of noncompact and compact type. We conclude the article by reviewing recent results \cite{OlafssonWolf2009} on the inductive limit of symmetric spaces.

\section{Fourier Analysis on $\Rn$ and the Euclidean Motion Group}\label{se-EmotGr}
\noindent
One of the simplest commutative spaces is $\Rn$ viewed as a homogeneous space for the Euclidean motion group. It is natural to ask how the Paley-Wiener theorem extends to this setting. We apply the discussion from the previous section to the commutative space $\Rn$, where the group is now the Euclidean motion group. We refer to \cite{OlafssonSchlichtkrullRT08} for some other aspects of this analysis.

Recall that the Euclidean motion group is $G=\SOn \ltimes \Rn$. View elements of $G$ as diffeomorphisms of $\Rn$ by
\[(A,x)\cdot y=A(y )+ x\, .\]
The multiplication in $G$ is a composition of maps: $(A,x)(B,y)=(AB, A(y)+x)$. The identity element is $(\rI,0)$, where $\rI$ is the identity matrix, and the inverse is $(A,x)^{-1}=(A^{-1}, -A^{-1}x)$. Let $K = \{ (A, 0) | A \in \SOn \}\simeq \SOn$. $K$ is the stabilizer of $0 \in \Rn$. Hence $\Rn \simeq G/K$. Note that $K$-invariant functions on $\Rn$ are radial functions, i.e., functions that only depend on $|x|$.

The regular action of $G$ on $ L^2(\Rn) $ is given by
$$ \ell_g f(y) = f(g^{-1} \cdot y) = f(A^{-1}(y-x)), \;\; g = (A,x)\, .$$

Put $L^2(\sn) = L^2(\sn, d\mu_n) $. For $r \in \R$ define a unitary representation $\pi_r$ of $G$ on $L^2(\sn) $ by
$$ \pi_r(A, x) \phi(\omega) := e^{-2\pi ir x \cdot \omega} \phi (A^{-1}(\omega)). $$
For $r \neq 0 $ the representation $ \pi_r$ is irreducible, and $\pi_r \simeq \pi_s$ if and only if $r = \pm s$. The intertwining operator is given by $[T f](\omega )=f(-\omega)$. Note that the constant function $p_r(\omega):=1$ on $\sn$ is a $K-$fixed vector for $\pi_r$. The corresponding vector valued Fourier transform, which we will also denote by $\cF_G(f)_r=\widehat{f}_r\in L^2(\sn)$, now becomes
\begin{align*}
    [\pi_r(f)p_r](\omega)  &= \int_G f(g) \pi_r(g) p_r(\omega) dg = \int_{G/K} f(x) \pi_r(x) p_r(\omega) dx \\
       &=\int_{\Rn} f(x) e^{-2\pi ir x \cdot \omega} dx = \FTrn f(r\omega)\, .
  \end{align*}
Let $d\tau (r)=\sigma_nr^{n-1}dr$. Then we have the following theorem:
\begin{thm}
The Fourier transform $f\mapsto \cF_G f$ extends to a unitary $G$-isomorphism
\begin{eqnarray*}
L^2(\Rn)&\simeq &  \int_{\R^+}^\oplus (\pi_r,L^2(\sn))\, d\tau (r)=L^2(\R^+,L^2(\sn);d\tau )\\
& \simeq & \{F\in L^2(\R ,L^2(\sn ); d\tau )\: F(r)(\omega )=F(-r)(-\omega )\}\, .
\end{eqnarray*}
The inverse is given by
\[f(x)=\int_0^\infty (\widehat{f}_r,\pi_r(x)p_r)\, d\tau (r)=\int_0^\infty\int_{\sn} \widehat{f}_r(\omega)e^{2\pi irx\cdot \omega}\, d\mu_n(\omega )d\tau (r)\, .\]
\end{thm}

\begin{proof} This follows from the Theorem \ref{th-PlancherelGK}.\end{proof}

The Hilbert space valued Paley-Wiener theorem, the Theorem \ref{secondTh}, describes the image of functions that are compactly supported and smooth in the radial variable. But if $F(r)$ is $\SO (n)$-finite, i.e., the translates $F(r)(k(\omega))$, $k\in\SO (n)$, span a finite dimensional space, then $\omega \mapsto F(r)(\omega )$ is a polynomial and hence has a holomorphic extension in the $\omega$-variable, showing that there is more in this than only the $L^2$-theory.

\section{Euclidean Paley-Wiener Theorem}\label{SectionEuclideanPW}
\noindent
In this section we discuss the Euclidean Paley-Wiener theorem with respect to the representations of the Euclidean motion group. We will give two different descriptions. Representations $\pi_r$ act on $L^2(\sn)$ and an instance of $L^2(\sn)$-valued functions is dealt with in the Paley-Wiener theorem proved in section \ref{VecValPW}. Note that the smooth vectors of the representation $\pi_r$ are the smooth functions on $\sn$: $L^2(\sn)^\infty = C^\infty (\sn)$. In this section we will work with functions valued in $C^\infty(\sn)$. We give the space $C^\infty(\sn)$ the Schwartz topology. With this topology it is equal to $\cE(\sn)=\SW(\sn)$. Since $\SW(\R)$ and $\SW(\sn)$ are nuclear spaces,
$$\SW(\R\times\sn)=\SW(\R, \SW(\sn)).$$
We will also denote it simply be $\C^{\infty}(\R\times\sn)$ or $C^{\infty}(\R, C^{\infty}(\sn))$ keeping in mind the Schwartz topology and that the first variable is related to the spectral decomposition. We will often identify these spaces algebraically and topologically by viewing functions $F:\R \to C^{\infty}(\sn)$ as functions $F:\R\times\sn \to \C$ via the mapping $F(z,\omega ):=F_z(\omega)$ and vice versa.

The first version is an analog of a variant due to Helgason, see Th. 2.10 in \cite{helgason2000}. We restate it here in a slightly different form.

For $r>0$, let $\PW_{r,H}^{\Z_2}(\C \times \sn)$ be the space of smooth functions $F$ on $\C \times \sn$ satisfying:
\begin{enumerate}
\item $F$ is even, i.e. $F(z,\omega )=F(-z,-\omega)$.
\item For each $\omega$, the function $z\mapsto F(z,\omega)$ is a holomorphic function on $\C$ with the property
\[|F(z,\omega )|\le C_N(1+|z|^2)^{-N}e^{r|\im z|} <\infty, \]
for each $N\in\N$.
\item For each $k\in \N$ and each isotropic vector $a\in\Cn$, the function
$$ z \mapsto z^{-k} \int\limits_{\sn} F(z,\omega )(a,\omega) d\omega $$
is even and holomorphic on $\Cn$.
\end{enumerate}

\begin{thm}\label{HelgDr} The Fourier transform followed by a holomorphic extension in the spectral parameter is an injection of $\cD_r(\Rn)$ onto $\PW_{2\pi r,H}^{\Z_2}(\C \times \sn)$.
\end{thm}

\begin{proof}
See \cite{helgason2000}, pages 23-28.
\end{proof}

Note that this theorem does not contain a topological statement. Next, we prove an analogous theorem for vector valued functions including the topological statement.

Define the space $\PW^{\Z_2}_{r,H}(\C, C^{\infty}(\sn))$ as the set of weakly holomorphic functions $F$ on $\C\times \sn$ which satisfy
\begin{enumerate}
\item $F$ is even, i.e. $F(r,\omega )=F(-r,-\omega)$.
\item For $k\in\N$, $\left(\frac{\partial}{\partial z}\right)^k F(z,\omega)|_{z=0}$ is a homogeneous polynomial of degree $k$ in $\omega$.
\item For $k\in\N$ and for any differential operator $D_{\omega}$ on the sphere
\[ |F|_{k,D_{\omega}}:=\sup\limits_{(z,\omega)\in\C\times\sn} (1+|z|^2)^k e^{-r|\im z|}   \left| D_{\omega} F(z,\omega )\right| < \infty .\]
\end{enumerate}
The topology on $\PW^{\Z_2}_{r,H}(\C, C^{\infty}(\sn))$ is given by the seminorms $|\cdot|_{k,D_{\omega}}$.

\begin{thm}\label{PW_Th} The Fourier transform $\cF_G$ followed by a holomorphic extension in the spectral parameter is a topological isomorphism of $\cD_r(\Rn)$ onto \newline $\PW^{\Z_2}_{2\pi r,H}(\C, C^{\infty}(\sn))$.
\end{thm}

\begin{proof}
Let $f\in\cD_r(\Rn)$. An analog of the Lemma \ref{lemWI} holds for Fr\'echet spaces and the first part of the proof of the Theorem \ref{secondTh} applies here. Thus $\cR f(x) e^{-2\pi izx}$ is weakly integrable and for every distribution $\Lambda$, $\Lambda(\cF_G f)$ is a holomorphic function. Let
$$F(z):=\cF_G f(z)= \int_{\R} \cR f(z,\omega) e^{-2\pi izx} dx.$$
By the Theorem 3.27 in \cite{rudin91}, $F:\C\mapsto C^{\infty}(\sn)$ is weakly holomorphic. Since $\cR f \in \cD_{H,r}^{\Z_2}(\R\times\sn)$, it is easy to see that all three conditions in the definition of $\PW^{\Z_2}_{2\pi r,H}(\C, C^{\infty}(\sn))$ are satisfied for $\cF_G f$ and the map $f\mapsto\cF_G f$ is continuous.

For the surjectivity part, let $F\in\PW^{\Z_2}_{2\pi r,H}(\C, C^{\infty}(\sn))$. Thus for any $\omega$, $r\mapsto\cF^{-1}_{\R}F(r,\omega) \in \cD_r(\R)$. It follows easily that $\cF^{-1}_{\R}F \in \cD_{H,r}(\Xi)$ and the map $F\mapsto\cF^{-1}_{\R}F$ is continuous. Hence $\cR^{-1}\cF^{-1}_{\R}F \in \cD_r(\Rn)$ and the map $F\mapsto\cF^{-1}_G F$ is continuous.
\end{proof}

\begin{rem}
This shows that the spaces in the two theorems are the same:
$$\PW^{\Z_2}_{r,H}(\C\times\sn) = \PW^{\Z_2}_{r,H}(\C, C^{\infty}(\sn)).$$
\end{rem}

For completeness we describe the image of the Schwartz functions under $\cF_G$.

Define the space $\cS^{\Z_2}_H(\R, C^{\infty}(\sn))=\cS^{\Z_2}_H(\R\times \sn)$ as the set of smooth functions $F$ on $\R\times \sn$ which satisfy
\begin{enumerate}
\item $F$ is even, i.e. $F(r,\omega )=F(-r,-\omega)$.
\item For $k\in\N$, $\left(\frac{\partial}{\partial r}\right)^k F(r,\omega)|_{r=0}$ is a homogeneous polynomial of degree $k$ in $\omega$.
\item For $k,l\in\N$ and for any differential operator $D_{\omega}$ on the sphere
\[ |F|_{k,l,D_{\omega}}:=\sup\limits_{(r,\omega)\in\R\times\sn} (1+|r|^2)^k \left|\left(\frac{\partial}{\partial r}\right)^l D_{\omega} F(r,\omega )\right| < \infty .\]
\end{enumerate}

With the topology given by the seminorms $|\cdot|_{k,l,D_{\omega}}$, the space $\cS^{\Z_2}_H(\R\times \sn)$ is Fr\'echet.

\begin{thm} The Fourier transform $\cF_G$ is a topological isomorphism of $\cS(\Rn)$ onto $\cS^{\Z_2}_H(\R\times \sn)$.
\end{thm}

\begin{proof}
Let $f\in \cS(\Rn)$ and let $F(r,\omega):=\widehat{f}_r(\omega)$. $F(r,\omega) = \cF_{\R} \cR f(r\omega)$. Clearly $F$ is even. Since $\cR f\in \cS_H(\Xi)$, it follows that  $\left(\frac{\partial}{\partial r}\right)^k F(r,\omega)|_{r=0}$ is a homogeneous polynomial of degree $k$ in $\omega$, as well as that for every $\omega$, $r\mapsto F(r,\omega) \in\cS(\R)$. By an application of the Lebesgue Dominated Convergence Theorem, it follows that for every $r$, $\omega \mapsto F(r,\omega)\in\cS(\sn)$. In particular, $F\in C^{\infty}(\R\times \sn)$. We also have
$$ (1+|r|^2)^k \left|\left(\frac{\partial}{\partial r}\right)^l D_{\omega} F(r,\omega )\right| \le \sum\limits_{finite} c_{N,\lambda}\, |\cF_{\Rn}f|_{N,\lambda}, $$
where $c_{N,\lambda}$ are some constants. Thus  $F\in \cS^{\Z_2}_H(\R\times \sn)$ and the mapping $f\mapsto F$ is continuous.

To show surjectivity, let $F \in \cS^{\Z_2}_H(\R\times \sn)$ and let $\varphi(r, \omega):=\cF^{-1}_{\R}F(r, \omega)$. The function $\varphi$ is Schwartz in both variables. Moreover,
$$ (1+|r|^2)^k \left(\frac{\partial}{\partial r}\right)^l D_{\omega} \varphi(r,\omega )=\sum\limits^k_{m=0} {k \choose m} i^{l-2m} \int \left(\frac{\partial}{\partial r}\right)^{2m} D_{\omega} F(s,\omega ) s^l e^{isr} ds$$
shows that the mapping $F\mapsto \varphi$ is continuous. Since
$$\int \varphi(r,\omega) r^k dr = c\,\left(\frac{\partial}{\partial r}\right)^k F(r,\omega)|_{r=0},$$
we get that $\varphi \in \cS_H(\Xi)$ and hence $\cR^{-1} \varphi \in \cS(\Rn)$. Clearly $\cF_G$ is injective. This proves the claim.
\end{proof}

For the second version of the Paley-Wiener theorem on $\cD_R(\Rn)$ we introduce first
\[S_\C^{n-1}:=\{z\in \C^n\: z_1^2+\ldots +z_n^2=1\}\]
the \textit{complexification} of $S^{n-1}$. It is easy to see that
\[S_\C^{n-1}=\SO (n,\C)/\SO (n-1,\C)\, .\]
Let $\C^{*}= \C \backslash \{0\}$, then the map $\C^{*} \times \snc \rightarrow \Cn \backslash \{z : \sum_{i=1}^n z_i^2 = 0 \} : (z,\omega) \mapsto z\omega $ is a holomorphic two-to-one map. Note that the Lebesgue measure of the set $ \{z \in \Cn : \sum_{i=1}^n z_i^2 = 0 \} $ is 0.

For $R>0, $ let $\cO_R(\C\times\snc)$ be the space of holomorphic functions $ F: \C\times\snc\to\C$
such  that for all $ N \in \N $
\begin{equation}\label{eq-pw1}
\PWN(F) := \sup\limits_{(z,\widetilde{\omega})\in \C\times\snc} (1 + |z \widetilde{\omega}|^2)^N e^{-R |\im(z \widetilde{\omega})|} |F(z, \widetilde{\omega})| < \infty.
\end{equation}
Since the space of holomorphic functions is nuclear, we can identify this space with the space of weakly holomorphic functions $\widetilde{F}:\C\to\cO(\snc)$ wich satisfy \ref{eq-pw1} by setting $\widetilde{F}(z)(\widetilde{\omega})=F(z,\widetilde{\omega})$.

The space of even functions $F \in \cO_R(\C\times\snc)$ satisfying that for all $\lambda\in\C$ and $\omega\in\sn$, $F(\lambda, \omega) = F(0, \omega) + \sum_{m=1}^{\infty} \frac{a_m(\omega)}{m!} \lambda^m $, where each $ a_m $ is a homogeneous polynomial in $ \omega_1, \dots ,\omega_n $ of degree $ m $, will be denoted by $\PWR = \PWR\left( \C \times \snc \right) = \PWR\left(\C,\cO(\snc) \right)$. The space $\PWR$ is a Fr{\'e}chet space.

Our aim is to prove the following:
\begin{thm}\label{firstTh}(Euclidean Paley-Wiener type theorem)
Let $f\in \cD_R(\Rn)$. Then $\cF_G f$ extends to an even holomorphic function on $\C\times \snc$, denote this extension by $\cF_G^c f$. Moreover, $\cF_G^c f \in \widetilde{\PW}_{H,2\pi R}^{\Z_2}$ and the map $\cD_R(\Rn)\to \widetilde{\PW}_{H,2\pi R}^{\Z_2}$: $f\mapsto \cF_G^c f$ is a topological isomorphism.
\end{thm}

For clarity of the exposition we will prove this result in several steps. We remark that by the Theorem \ref{th-1.2}, it suffices to prove the following: for $\varphi \in \cD_{H,R}(\Xi)$ the Fourier transform $\cF_{\R}(\varphi)(r,\omega)$ extends to an even holomorphic function on $\C\times\snc$, this extension belongs to $\widetilde{\PW}_{H,2\pi R}^{\Z_2}$, and $\cF_{\R}^c$ defines a topological isomorphism $\cD_{H,R}(\Xi)\simeq \widetilde{\PW}_{H,2\pi R}^{\Z_2}$.

\begin{lem}\label{PWdiffPW}
Let $k \in \N$. Then the map $ F \mapsto \left( \frac{d}{d\xi} \right)^k  F $ is a linear continuous mapping from $\PWR $ into itself.
\end{lem}
\begin{Proof}
Let $F \in \PWR$, $ \xi \in \C $ and $ \widetilde{\omega} \in \snc $ be fixed. For some $ \delta > 0 $, let $\gamma(t) = \xi + \delta e^{it}$, with $0 \leq t \leq 2\pi $. Then \[\frac{\partial}{\partial \xi} F(\xi, \widetilde{\omega}) = \frac{1}{2\pi i} \int_{\mathscr{\gamma}} \frac{F(z, \widetilde{\omega}) }{(z-\xi)^2} \; dz \, .\]
Note that this holds for any $ \delta > 0 $. We have
$$\begin{array}{rcl}
          (1 + \left| \xi \widetilde{\omega} \right|^2)^N \left|\frac{d}{d\xi} F(\xi, \widetilde{\omega}) \right| \!\!\!\!\!\!\!\!\!\!\!\!\! && e^{-R \left|\im(\xi \widetilde{\omega})\right|} \\
		& \leq & (1 + \left|\xi \widetilde{\omega} \right|^2)^N e^{-R \left|\im(\xi \widetilde{\omega})\right|} \frac{1}{2\pi} \int\limits_{\mathscr{\gamma}} \frac{ \left|F(z, \widetilde{\omega})\right| }{\left|z-\xi\right|^2} \left|dz\right|\\
 		&\leq & \frac{1}{2\pi \delta^{2}} \int\limits_{\mathscr{\gamma}} \frac{(1 + \left| \xi \widetilde{\omega} \right|^2)^N (1 + \left| z \widetilde{\omega} \right|^2)^N \left| F(z, \widetilde{\omega}) \right| e^{-R \left| \im(\xi \widetilde{\omega}) \right|}}{(1 + \left| z \widetilde{\omega} \right|^2)^N} \left| dz \right|.
        \end{array} $$
Further observe that (i) $| \re(\widetilde{\omega}) |^2 = | \im(\widetilde{\omega}) |^2 + 1,$ which implies $| \im(\widetilde{\omega}) | + 1 \geq | \re(\widetilde{\omega})|$, (ii) $\left| \im(\xi \widetilde{\omega}) \right|^2 = |\im(\xi)|^2 + |\xi|^2 |\im(\widetilde{\omega})|^2 =  |\xi|^2 |\re(\widetilde{\omega})|^2 - |\re(\xi)|^2$, and (iii) $|z| \leq |\xi| + \delta$. Applying (i) to (iii) gives: $ |\im(\xi \widetilde{\omega})| \geq \left| \im(z \widetilde{\omega}) \right| - \delta |Re(\widetilde{\omega})| - |\xi|$. Hence $ e^{-R \left| \im(\xi \widetilde{\omega}) \right|} \leq e^{-R \left| \im(z \widetilde{\omega}) \right|} e^{R \delta |Re(\widetilde{\omega})|} e^{R |\xi|} $. Since $\xi$ is fixed, the last exponential is some positive constant $\geq 1$, call it $C$, and by choosing $ \delta < \frac{1}{|\widetilde{\omega}|} $, we obtain: $ e^{-R \left| \im(\xi \widetilde{\omega}) \right|} \leq C e^{-R \left| \im(z \widetilde{\omega}) \right|} e^R.$ Thus,
$$\begin{array}{rcl}
          (1 + \left| \xi \widetilde{\omega} \right|^2)^N \left|\frac{d}{d\xi} F(\xi, \widetilde{\omega}) \right| \!\!\!\!\!\!\!\!\!\!\!\!\! && e^{-R \left|\im(\xi \widetilde{\omega})\right|} \\
		& \leq & \frac{C e^R}{2\pi \delta^{2}} \int\limits_{\mathscr{\gamma}} \frac{(1 + \left| \xi \widetilde{\omega} \right|^2)^N (1 + \left| z \widetilde{\omega} \right|^2)^N \left| F(z, \widetilde{\omega}) \right| e^{-R \left| \im(z \widetilde{\omega}) \right|}}{(1 + \left| z \widetilde{\omega} \right|^2)^N} \left| dz \right| \\
 		&\leq & \frac{ C e^R \PWN(F)}{2\pi \delta^{2}} \int\limits_{\mathscr{\gamma}} \frac{(1 + \left| \xi \widetilde{\omega} \right|^2)^N }{(1 + \left| z \widetilde{\omega} \right|^2)^N} \left| dz \right|.
\end{array} $$
Next note $$\frac{1 + \left| \xi \widetilde{\omega} \right|^2}{1 + \left| z \widetilde{\omega} \right|^2} \leq \frac{1 + |\widetilde{\omega} |^2 \left| \xi \right|^2 }{1 + \left| z \right|^2}  = \frac{1 + |\widetilde{\omega} |^2 \left| z - \delta e^{it} \right|^2 }{1 + \left| z \right|^2}  \leq \frac{|z|^2 + 2  |z| + 2}{(1 + \left| z \right|^2)} \leq 5, $$ where we used that  $ \left| z - \delta e^{it} \right|^2 \leq |z|^2 + 2 \delta |z| + \delta^2 $ and $\delta \leq 1$. Hence $ \PWN(\frac{d}{d\xi} F) \leq \frac{C e^R 5^N}{\delta} \PWN(F) $. It is easy to see that $ \frac{d}{d\xi} F $ is even and satisfies the homogeneity condition. Hence $ \frac{d}{d\xi}F \in \PWR $ and the map $F \mapsto \frac{d}{d\xi}F $ is continuous. Iterating this argument the statement follows. \qed
\end{Proof}

\begin{coro}\label{PWrestSW}
Let $ \omega \in \sn $. The restriction map $R_{\omega}: \cO_R(\C\times\snc) \rightarrow \SW(\R): F \mapsto F(\cdot, \omega)|_{\R}$ is a continuous linear transformation.
\end{coro}

\begin{Proof}
By the above proof, restricting $F$ to $ \R \times \sn $ yields
$$ (1 + |r|^2)^N \, \left|\left( \frac{d}{dr} \right)^k F(r, \omega)\right| \leq  \left( \frac{C e^R 5^N}{\delta}  \right)^k \PWN(F). $$
Thus, if $f(r):=F(r, \omega)$ with  $\omega \in \sn$ fixed, then $ |f|_{N,k} \leq  \left( \frac{C e^R 5^N}{\delta}  \right)^k \PWN(F)$. \qed
\end{Proof}
A simple application of the Cauchy's Integral Formula gives:
\begin{lem}\label{PWmovePI}
Let $F\in \cO_R(\C\times\snc)$ and $\omega \in \sn $. Then for any $ y \in \R $, $ \int\limits_{\R} F(t, \omega) dt = \int\limits_{\R} F(t+iy, \omega) dt $.
\end{lem}
\noindent
It is also not hard to see that:
\begin{lem}\label{PWtEXP}
Let $ F \in \cO_R(\C\times\snc)$ and $ r\in \R $. Define $ H (\xi, \widetilde{\omega} ):= F(\xi, \widetilde{\omega}) e^{i \xi r} $. Then $ H \in \PWRP $.
\end{lem}
\noindent
We now complete the proof of the Theorem \ref{firstTh}:
\begin{Proof}
Let $ \varphi \in \cD_{H,R}(\Xi)$ and $ \omega \in \sn $ be fixed. Then for $ \xi = x + iy \in \C $ and $ r \in [-R, R] $ we have the estimate $ |e^{-2\pi ir\xi}| \leq e^{2\pi R |\im(\xi)|} $. Hence $$ \left| \int_{-R}^R \varphi(r, \omega) e^{-2\pi ir\xi} dr \right| \leq e^{2\pi R |\im(\xi)|} | \varphi( \cdotp, \omega) |_{L^1} \leq 2R \, | \varphi( \cdotp, \omega) |_{\infty} \, e^{2\pi R |\im(\xi)|} < \infty.$$ This shows that for every $ \xi \in \C, \; \cF_{\R} (\varphi)( \xi, \omega) $ is well-defined. Let $  \xi_o \in \C $, and $ \epsilon > 0$,  then for each $\xi \in \{z : | z-\xi_o | < \epsilon \} $ we have the estimate
$$ | \varphi(r, \omega) e^{-2\pi ir\xi}| \leq  | \varphi( \cdotp, \omega) |_{\infty} \, \chi_{[-R,R]}(r) \, e^{2\pi |r| |\im(\xi)|} \in L^1(\R).$$
Thus $ \xi \mapsto \cF_{\R} (\varphi)( \xi, \omega) $ converges uniformly on compact subsets of $\C$ and hence is holomorphic. Moreover
$$ \frac{d}{d\xi} \cF_{\R} (\varphi)( \xi, \omega) = \int_{\R} \varphi(r, \omega) \frac{d}{d\xi} e^{-2\pi ir\xi} dr. $$

\noindent
Since  $ \varphi \in \cD_{H,R}(\Xi) $, there is $ f \in \DS_R(\Rn) $ such that $ \RT f = \varphi $. Define a function $F$ by
\begin{align*}
     F(r, \omega) :&= \cF_{\R} (\varphi)( r, \omega) \\
       				&= \int_{-\infty}^{\infty}  \mathcal{R}f(s, \omega) \, e^{-2\pi isr} ds \\
       				&= \int_{\Rn}  f(x) \, e^{-2\pi ir\omega \cdot x} dx \\
       				&= \FTrn (f)(r\omega).
  \end{align*}
The penultimate equality holds by the Fourier-Slice theorem. Now, by the classical Paley-Wiener theorem $ \FTrn(f) $ has a holomorphic extension to $ \C^n $.
It follows that $ F $ extends to a holomorphic function on $ \C \times \snc $:
$$\begin{array}{lclcl}
          \C \times \snc 	& \longrightarrow	& \C^n		& \longrightarrow & \C \\
          (z, \widetilde{\omega})	& \longmapsto	& z \widetilde{\omega} & \longmapsto	& \FTrnc(f)(z \widetilde{\omega}) =: F(z, \widetilde{\omega}),\\
\end{array} $$
and we have the estimate
\begin{equation}\label{est2}
	\sup\limits_{(z,\widetilde{\omega})\in\C\times\snc} (1 + |z \widetilde{\omega}|^2)^N e^{-2\pi R |\im(z \widetilde{\omega})|} |F( z, \widetilde{\omega} ) | < \infty.
\end{equation}
Note that we have holomorphically extended $ \cF_{\R} (\varphi) $ in two different ways to two different domains, namely to $  \C \times \sn $ and to $  \C \times \snc $. It is easy to verify that these two extensions agree on the common domain. Since $\sn$ is a totally real submanifold of $\snc$, to show that $ F(-\xi, -\widetilde{\omega}) = F(\xi, \widetilde{\omega}) $, with $ \xi \in \C $ and  $ \widetilde{\omega} \in \snc $, it is enough to verify it for $ \omega \in \sn $, which is easy.

Let $a_k(\omega) := \left( \frac{d}{d\xi} \right)^{k}  F(\xi, \omega) |_{\xi=0} $. As we can differentiate inside the integral, $a_k(\omega) =(-2\pi i)^k \int_{\R} \varphi(r, \omega) r^k dr $. Thus for $k \in \N^+$, $ a_k $ is a homogeneous polynomial in $ \omega_1, \dots, \omega_n $ of degree $k$. Hence for $ \xi \in \C $, $ \omega \in \sn $
$$ F(\xi, \omega) = F(0, \omega) + \sum\limits_{m = 1}^{\infty} \frac{\left( \frac{d}{d\xi} \right)^{m}  F(\xi, \omega) |_{\xi=0}}{m!} \xi^m = F(0, \omega) + \sum\limits_{m = 1}^{\infty} \frac{a_m(\omega)}{m!} \xi^m. $$
This shows that $ \ftc (\varphi) := F \in \widetilde{\PW}_{H,2\pi R}^{\Z_2}$. The map is injective and since the Radon transform is a linear topological isomorphism of $\DS_R(\Rn)$ with $\cD_{H,R}(\Xi)$ \cite{hertle84}, it is also continuous.

To show surjectivity, let $ F \in \widetilde{\PW}_{H,2\pi R}^{\Z_2}$. For $ \omega \in \sn $, the map $\R \rightarrow \C: r \mapsto F(r, \omega)$ is a Schwartz function by the Corollary \ref{PWrestSW}. We will use the same letter for this restriction of $F$.  Since the Fourier transform is a topological isomorphism of the Schwartz space with itself, $ \cF_{\R}^{-1}(F)$ is a Schwartz function in the first variable, call it $\varphi$.
By the Lemmas \ref{PWtEXP} and \ref{PWmovePI}, it follows that for any $ y \in \R $ and any $\omega \in \sn$, we have: $ \varphi(r, \omega) = \int_{\R} F(x + iy, \omega) \, e^{2\pi i(x + iy)r} dx $.

Let $ \eta \in \R $, then $\varphi(r, \omega)=e^{-2\pi \eta r}  \int_{\R} F(x+i\eta, \omega) \, e^{2\pi i x r} dx$. Since for all $N \in \N $ we have $ | F(x+i\eta, \omega) | \leq c (1 + x^2 + \eta^2)^{-N} e^{2\pi R |\eta|} $ for some constant $c$, it follows that
\begin{align*}
      | \varphi(r, \omega) | &\leq e^{-2\pi \eta r}  \int_{\R} | F(x+i\eta, \omega) | dx \\
       &\leq e^{2\pi (R |\eta| - \eta r)} c \int_{\R} (1 + x^2)^{-N} dx.
  \end{align*}
Take $N$ big enough so that the last integral is finite and let $ |\eta| \rightarrow \infty $. We obtain $ \varphi(r, \omega) = 0 $ for $ |r| > R $. Hence supp$(\varphi) \subseteq [-R, R] \times \sn $, and $ r \mapsto \varphi(r, \omega) \in  \mathcal{D}_R \left( \R \right) $.

To show that for any $x \in \R $ the function $\omega\mapsto\varphi(x, \omega)$ is $ C^{\infty}(\sn)$, we have to show $ \left| D^{\alpha}_{\omega} F(r, \omega) \right| \leq |f(r)|$ for some integrable function $f$ and any multi-index $ \alpha \in \N^n $. Then by the Lebesgue Dominated Convergence theorem
$$ D^{\alpha}_{\omega} \varphi(x, \omega) = \int_{\R} D^{\alpha}_{\omega} F(r, \omega) \, e^{2\pi irx} dr $$
and we are done. It is enough to show it for $ D^{\alpha}_{\omega} = \frac{\partial}{\partial \omega_j } $ for some $j \in \{ 1, \dots, n \} $, and then argue inductively.

Fix $r \in \R$ and $\omega \in \sn $, and let $\gamma(t) = \omega + \delta e^{it} e_j$, with $0 \leq t \leq 2\pi $, then
$$ \left| \frac{\partial}{\partial \omega_j} F(r, \omega) \right| \leq (2 \pi)^{-1}  \oint\limits_{\gamma} \frac{\left|  F(r, \xi) \right|}{\left| \xi-\omega \right|^2} |d\xi| = (2 \pi)^{-1} \delta^{-2} \oint\limits_{\gamma} \left|  F(r, \xi) \right| |d\xi|.$$
Since this holds for any $ \delta > 0 $, we can choose $ \delta < \frac{1}{1 + |r|} $. Note that $ 1 + |r\xi|^2 \geq 1 + |r|^2 $, and
$ |\im(\xi)| \leq \delta c$ for some constant $c>0$. This gives $$ |F(r, \xi)| \leq \PWN(F) (1+|r|^2)^{-N} e^{2\pi Rr\delta c} \leq \PWN(F) (1+|r|^2)^{-N} e^{2\pi R c} $$ for all $N \in \N $. Hence
$$ \left| \frac{\partial}{\partial \omega_j} F(r, \omega)\right| \leq \frac{\PWN(F) e^{2\pi R c}}{\delta} (1+|r|^2)^{-N} \in L^1_r(\R) \text{  for $N$ big enough.} $$

As $x\mapsto F(x, \omega)$ is Schwartz, we can differentiate inside the integral in
$$\left( \frac{d}{dr} \right)^k \varphi(r, \omega)=\int_{\R} F(x, \omega) \left( \frac{d}{dr} \right)^k e^{2\pi i x r} dx.$$
Furthermore, since for every $k\in \N$ and $\alpha \in \N^n$,
$ \left| D^{\alpha}_{\omega} F(x, \omega) (2\pi ix)^k \right| \leq C\frac{|x|^k}{(1+|x|^2)^{-N} }$ and $\frac{|x|^k}{(1+|x|^2)^{-N}} \in L^1_r(\R)$ for $N$ big enough, we have 
$$ \left( \frac{d}{dr} \right)^k  D^{\alpha}_{\omega}  \varphi(r, \omega)= \int_{\R}  D^{\alpha}_{\omega} F(x, \omega) \left( \frac{d}{dr} \right)^k e^{2\pi i x r} dx.$$
Thus we have the estimate,
$$ \left| \left( \frac{d}{dr} \right)^k  D^{\alpha}_{\omega}  \varphi(r, \omega) \right| \leq \frac{e^{2\pi R c}}{\delta^{|\alpha|}} \PWN(F) \int_{\R} \frac{|x|^k}{(1+|x|^2)^{-N}} dx, $$ for any $N \in \N $. Choosing $N$ big enough, we obtain $ \left| \left( \frac{d}{dr} \right)^k  D^{\alpha}_{\omega}  \varphi(r, \omega) \right| \leq \widetilde{c} \PWN(F), $ for some constant $\widetilde{c}$. Thus we conclude, $ \nb_{k, D}(\varphi) < \infty$ for any $ k \in \N $ and for any $D_{\omega}$, a differential operator on the sphere. Moreover this shows that the inversion is continuous.

By assumption, for $k \in \N^+$, $ \left( \frac{d}{d\xi} \right)^{k}  F(\xi, \omega) |_{\xi=0} $ is a homogeneous polynomial of degree $k$ in $ \omega_1, \dots, \omega_n $. Since $  \int_{\R} \varphi(r, \omega) r^k dr = \frac{1}{(-2\pi i)^k}  \left( \frac{d}{dr} \right)^{k}  F(r, \omega) |_{r=0} $, $  \varphi $ satisfies the homogeneity condition. It is easy to see that $\varphi(-r, -\omega) = \varphi(r, \omega)$. Thus, $  \varphi \in \cD_{H,R}(\Xi) $. \qed
\end{Proof}

\begin{rem}
For $F\in\widetilde{\PW}_{H,2\pi R}^{\Z_2}$, let the $Ext(F)$ denote the extension of $F$ to the whole $\Cn$. It is easy to see that $Ext$ is injective and continuous. We have the following commutative diagram:
$$\begin{CD}
\DRH @>\ftc>> \widetilde{\PW}_{H,2\pi R}^{\Z_2}(\C\times\snc)\\
@A\RT AA @VV Ext V\\
\DS_R(\Rn) @>\FT_{\Rn} >> \PW_{2\pi R}(\Cn)
\end{CD}$$
Since the Fourier transforms $\ftc$ and $\FT_{\Rn}^{-1}$, as well as the Radon transform $\RT$, are linear topological isomorphisms between the function spaces indicated in the diagram, it follows that the extension map: $$ Ext: \widetilde{\PW}_{H,2\pi R}^{\Z_2}(\C\times\snc) \rightarrow \PW_{2\pi R}(\Cn)$$ is a linear topological isomorphism.

We can view this in a different way. Let us re-draw the above diagram as follows:
\begin{equation}\label{eq-ComDia1}
\begin{CD}
\DS_R(\Rn)  @> \RT >> \DRH \\
@V\FTrnc VV 								@VV\ftc V\\
\PW_{2\pi R}(\Cn) @> \widetilde{\RT}>> \widetilde{\PW}_{H,2\pi R}^{\Z_2}(\C\times\snc)
\end{CD}
\end{equation}
We obtain a Radon type transform $\widetilde{\RT}$ between the spaces $\PW_{2\pi R}(\Cn)$ and $\widetilde{\PW}_{H,2\pi R}^{\Z_2}(\C\times\snc)$. For a function $F\in\PW_{2\pi R}(\Cn)$ there is a unique function $f\in\DS_R(\Rn)$ such that $\FTrnc f=F$, and $\widetilde{\RT} F $ is defined as:
$$ \widetilde{\RT} F (z, \widetilde{\omega}) := \ftc(\RT f)  (z, \widetilde{\omega}) = \FTrnc f  (z\widetilde{\omega})=F(z\widetilde{\omega}).$$
\end{rem}

\begin{rem}
Let $F$ be a function in $\PWR\left(\C,\cO(\snc)\right)$. Consider its restriction to the sphere: $F|_{\sn}$. Clearly $F|_{\sn}$ is in $\PW^{\Z_2}_{R,H}(\C, C^{\infty}(\sn))$ and this restriction map is injective. By the two Theorems \ref{firstTh} and \ref{PW_Th}, it is also surjective.

Some results of this flavor have been obtained in \cite{DamelinDevaney2007}. There a local Paley-Wiener theorem is considered, and the authors give necessary and sufficient conditions for a function, which restricts analytically to the sphere $\sn$ with $n=2,3$, to be in a Paley-Wiener space, $\Hr(\Cn)$ for some $r>0$.
\end{rem}

\section{Semisimple Symmetric Spaces of Noncompact Type}\label{SectionSSSofNT}
\noindent
In this section we recall the Helgason-Gangolli Paley-Wiener theorem for Rie\-man\-nian symmetric spaces of noncompact type. Recall that a Riemannian symmetric space $X=G/K$ is called to be of \textit{noncompact type} if $G$ is a connected noncompact semisimple Lie group with a finite center and without compact factors, and $K$ is a maximal compact subgroup. Then $X$ is a commutative space.
The Paley-Wiener theorem was extended to this setting by Helgason and Gangolli \cite{gangolli71,helgason66}. The proof was later simplified by Rosenberg \cite{rosenberg77}. In essence the theorem says that $\lambda\mapsto \widehat{f}_\lambda=\widehat{f} (\lambda )$ extends to a holomorphic function in the spectral parameter $\lambda$ and this extension is of exponential growth $r$ if and only if $f$ is supported in a ball of radius $r$ centered at the base point $eK$. Furthermore, the Fourier transform $\widehat{f}$ satisfies intertwining relations coming from the equivalence of the representations $\pi_\lambda$.

Let $\theta: G\to G$ be the Cartan involution corresponding to the maximal compact subgroup $K$. Denote the corresponding involution on the Lie algebra by the same letter. Then $\fg=\fk\oplus \fs$, where $\fk=\fg^\theta =\{X\in \fg\: \theta (X)=X\}$ is the Lie algebra of $K$, and $\fs=\fg^{-\theta}=\{X\in\fg\: \theta(X)=-X\}$ corresponds to the tangent space of $X$ at the base point $x_o=eK$. Fix a $K$-invariant inner product $(\, ,\, )$ on $\fg$, i.e.,  $X,Y\mapsto  -\Tr (\ad (X)\ad (\theta(Y)))$. Then $(\, ,\, )$ defines an inner product on $\theta$-invariant subspaces of $\fg$ and also a Riemannian structure on $X$.

Let $\fa\subset \fs$ be a maximal abelian subspace and $\Sigma \subset \fa^*$ the set of (restricted) roots. For $\alpha\in \Sigma$, let $\fg_\alpha:=\{X\in \fg\: (\forall H\in\fa)\,\, [H,X]=\alpha (H)X\}$ be the corresponding root space. As $\fa_r=\{H\in\fa\: (\forall \alpha\in\Sigma)\,\, \alpha (H)\not= 0\}$ is open and dense in $\fa$, there exists $H_o\in\fa$ such that
$\alpha (H_o)\not= 0$ for all $\alpha\in\Sigma$. Let $\Sigma^+=\{\alpha\in\Sigma\: \alpha (H_o)>0\}$. As $\theta (\fg_\alpha)=\fg_{-\alpha}$, it follows that $\Sigma$ is invariant under the multiplication by $-1$. In particular, $\Sigma = \Sigma^+\dot{\cup} -\Sigma^+$. Let $\fn:=\bigoplus_{\alpha\in\Sigma^+}\fg_\alpha$ and $\bar{\fn}:=\theta (\fn)=\bigoplus_{\alpha\in-\Sigma^+}\fg_\alpha$. Finally, let $\fm:=\fz_\fk (\fa)=\{X\in \fk\: [X,\fa]=\{0\}\}$ and $\fp:=\fm\oplus\fa\oplus \fn$. Then $\fp$ is a Lie algebra. Define $P=N_G(\fp)$, $M=Z_K(\fa)$, $A=\exp \fa$ and $N=\exp \fn$. Then $P=MAN$ and the multiplication map $M\times A\times N\to P$, $(m,a,n)\mapsto man$, is a diffeomorphism. Furthermore, the group $MA$ normalizes $N$. We also have the Iwasawa decomposition $G=NAK=KAN\simeq K\times A\times N$. In particular, $G=KP$ and $G/P=K/M$. We set $B=K/M$. We write $x=k(x)a(x)n(x)$, where $x\mapsto (k(x),a(x),n(x))\in K\times A\times N$ is an analytic diffeomorphism. Note that all of these maps are well defined on $B$. The action of $G$ on $B$ is then given by $k\cdot b=k(b)$.

Let $W:=N_K(\fa) /M$. Then $W$ is a finite reflection group, the (little) Weyl group. It is generated by the reflections
\[H\mapsto s_\alpha (H)=H-\alpha (H)H_\alpha\, ,\]
where $H_\alpha\in [\fg_\alpha, \fg_{-\alpha}]$ is such that $\alpha (H_\alpha )=2$.

For $\lambda\in\fa^*_\C$  and $a=\exp H\in A$ set $a^\lambda :=e^{\lambda (H)}$. Then $a\mapsto a^\lambda$ is a character on $A$. It is unitary if and only if $\lambda\in i\fa^*$. Let $m_\alpha :=\dim \fg_\alpha$,  $\alpha\in\Sigma$, and define $\rho:=\frac{1}{2}\sum_{\alpha\in\Sigma^+}m_\alpha \alpha$. Note, even if we don't use it, that $\rho $ can be viewed as an element of $(\fm\oplus \fa\oplus \fn)^*$ by
$\rho = \frac{1}{2}\Tr (\ad|_{\fn})$.

Define a representation of $G$ on $L^2(B)$ by
\[\pi_\lambda (x)f( b)= a(x^{-1}k)^{\lambda -\rho}f(x^{-1}\cdot b)\, ,\]
with $b=k\cdot x_0 \in K/M$.
The representations $(\pi_\lambda, L^2(B))$ are the \textit{principal series representations}.
$\pi_\lambda$ is unitary if and only if $\lambda\in i\fa^*$ and $\pi_\lambda$ is irreducible for almost all $\lambda\in \fa_\C^*$.
$\pi_\lambda$ is equivalent to $\pi_\mu$ if and only if there exists $w\in W$ such that $w\lambda = \mu$.  The function $p_\lambda =1$ is clearly $K$-invariant.  We normalize the intertwining operator $\cA(w,\lambda ): L^2(B)\to L^2(B)$ such that $\cA (w,\lambda )p_\lambda=p_{w\lambda}$.

We note that the Hilbert space $L^2(B)$ is the same for each of the representations $\pi_\lambda$. Thus, if $\mu$ is a measure on $i\fa^*$ and $\Lambda\subseteq i\fa^*$ is measurable, then
\[\int^\oplus_{\Lambda}(\pi_\lambda ,L^2(B))\, d\mu (\lambda )\simeq
L^2(\Lambda ,L^2(B);\mu (\lambda))\simeq L^2(\Lambda ,\mu )\overline{\otimes}
L^2(B)\]
where $\overline{\otimes}$ denotes the Hilbert space tensor product. If $\varphi$ is a section in the direct integral, then we write $\varphi_\lambda$ or $\varphi (\lambda )$ for the $\varphi$ evaluated at $\lambda$.

For $f\in C_c^\infty (X)$ the Fourier transform is now
\begin{eqnarray*}
\widehat{f}_\lambda (b) &= &\int_{X}f(x)\pi_\lambda (x)p_\lambda (b)\, dx\\
&=&\int_{X}f(x)a(x^{-1}b)^{\lambda -\rho}\, dx\\
&=&\int_{X}f(x) e_{-\lambda,b}(x)\, dx
\end{eqnarray*}
where $\displaystyle{e_{\lambda,b}(x):=a(x^{-1}b)^{-\lambda -\rho}}$.
Thus, the vector valued Fourier transform $\widehat{f}_\lambda$ evaluated at $b\in B$ is exactly the Helgason Fourier transform on $X$.  Note however, that our notation differs from that of Helgason by an $i$ in the exponent. It differs from \cite{OlafssonSchlichtkrullRT08} by a minus sign. This is done so that it fits better to the compact case which we will discuss in a moment.

We have
\begin{eqnarray}
\cA (w,\lambda )\widehat{f}_\lambda &=&\int_{X} f(x)\cA (w,\lambda )[\pi_\lambda
(x)p_\lambda]\, dx\nonumber\\
&=&\int_{X} f(x)\pi_{w\lambda}(x)p_{w\lambda}\, dx\nonumber\\
&=&\widehat{f}_{w\lambda}\, .\label{IntertwiningRelation}
\end{eqnarray}

If $f$ is $K$-invariant, then $\widehat{f}_\lambda$ is independent of $b$ and we simply write $\widehat{f} (\lambda )$ for $\widehat{f}_\lambda (b)$. We have
\[\widehat{f}(\lambda )=\int_X f(x)\left(\int_K a(x^{-1}k)^{\lambda -\rho}\, dk\right)\, dx
= \int_X f(x)\varphi_{-\lambda } (x)\, dx\]
where $\varphi_\lambda $ denotes the spherical function
\begin{equation}\label{eq-sphericalFunction}
x\mapsto (\pi_{-\lambda } (x)p_{-\lambda },p_{-\lambda} )=
\int_K a(x^{-1}k)^{-\lambda - \rho}\, dk\, .
\end{equation}
We have $\varphi_\lambda =\varphi_\mu$ if and only if $\lambda \in W\cdot \mu$ and the intertwining relation (\ref{IntertwiningRelation}) reduces to $\widehat{f}(\lambda )=\widehat{f}(w\cdot \lambda )$.

Let $c(\lambda )$ be the Harish-Chandra $c$-function. We won't need the exact form here, but recall that it can be expressed as a multiple of $\Gamma$-functions \cite{GindikinKarpelevic62}. Define a measure on $i\fa^*$ by $d\mu_X (i\lambda )=(\# W|c (\lambda )|^2)^{-1}d(i\lambda)$. Let
\[L^2_W(i\fa^*,L^2(B);\mu_X):=\{F \in L^2(i\fa^*,L^2(B);\mu_X)\: \cA (w,\lambda )F_\lambda = F_{w\lambda}\}\, .\]
Then $L^2_W(\fa^*,L^2(B);\mu_X)$ is a closed subspace of $L^2(i\fa^*,L^2(B);\mu_X)$ and hence a Hilbert space.

\begin{thm} The Fourier transform extends to a unitary isomorphism
\[L^2(X)\simeq L^2_W(i\fa^*,L^2(B);\mu_X)\, .\]
\end{thm}
\begin{proof} See \cite{helgason2008}, p. 202.
\end{proof}

For $r>0$, let $\overline{B}_r(x_o)$ denote the closed ball of radius $r>0$ and center $x_o=eK$. Let $\PW_{r,W}(\fa_\C^*,C^\infty (B))$ denote the space of holomorphic functions $F:\fa_\C^*\to C^\infty (B)$ such that
\begin{enumerate}
\item For each $N\in\N$, and $D$ a differential operator on $B$ we have
$$\displaystyle{\sigma_{N}(F):=\sup_{\lambda\in\fa_\C^*} (1+|\lambda |^2)^Ne^{-r|\im \lambda |} \|DF(\lambda )\|_\infty<\infty} ,$$
\item $\displaystyle{ \cA (w,\lambda )F(\lambda )=F(w\lambda)}$.
\end{enumerate}
The topology defined by the seminorms $\sigma_N$ turns $\PW_{r,W}(\fa_\C^*,C^\infty (B))$ into a Fr\'echet space.

\begin{thm} If $f\in C_r^\infty (X)$, then $\widehat{f}$ extends to a holomorphic function $\widehat{f}^c$ on $\fa_\C^*$, $\widehat{f}^c\in \PW_{r,W}(\fa_\C^*,C^\infty (B))$ and $f\mapsto \widehat{f}^c$  is a topological isomorphism $C^\infty_r(X)\simeq \PW_{r,W}(\fa_\C^*,C^\infty (B))$.
\end{thm}
\begin{proof} See \cite{gangolli71,helgason66,helgason2008,rosenberg77}. For the formulation as above, see \cite{danielsen2010}.
\end{proof}

An important step in the proof is the generalization of the Fourier-Slice Theorem (\ref{eq-FST}). For that let us recall the Radon transform for $X$. The \textit{horocycles} in $X$ are the orbits of the group $N$. Using that $G=NAK$, it follows easily that each horocycle is of the form $\xi (kM,a)=kaN\cdot x_o$, and that $\Xi$, the space of horocycles, is a $G$-space and isomorphic to $G/MN\simeq K/M\times A$.
The Radon transform of a function $f\in C_c^\infty (X)$ is given by
\begin{equation}\label{def-RadTrGK}
\cR (f)(kM,a):=\int_{N}f(kan\cdot x_o)\, dk\, .
\end{equation}

For $r>0$, denote by $C^\infty_r(\Xi)$ the space of smooth functions $\varphi$ on $\Xi$ such that $\varphi (b,a)=0$ for $|\log a|\ge 0$. Then
\begin{equation}\label{RadonSupportThm}
\cR (C_r^\infty (X))\subseteq C^\infty_r(\Xi )\, .
\end{equation}
Furthermore, there exits a constant $c>0$ such that
\begin{equation}\label{eq-FSl2}
\widehat{f}(\lambda ,b)=c\cF_A (\cR (f))(\lambda ,b)
\end{equation}
where $\cF_A$ stands for the Fourier transform on the vector group $A\simeq \fa$.
The Eu\-cli\-de\-an Paley-Wiener The\-orem now implies that $\widehat{C^\infty_r(X)}\subseteq \PW_{r,W}(\fa_\C^*,C^\infty (B))$.


\section{Semisimple Symmetric Spaces of the Compact Type}\label{SectionSSSofCT}
\noindent
Now we discuss the Paley-Wiener theorem for symmetric spaces of compact type.
The case of central functions on compact Lie groups $U\simeq U\times U/\text{diag}(U)$ was considered by Gonzalez in \cite{gonzalez2001}. The general case of $K$-invariant functions on $U/K$ was treated in \cite{BransonOlafssonPasqualeH,BransonOlafssonPasquale2005,OlafssonSchlichtkrullPW08,OlafssonSchlichtkrullD}. The first two article considered only the case of even multiplicities $m_\alpha$. The $K$-finite case was solved in \cite{OlafssonSchlichtkrull2010}. The case of $K$-invariant functions on the sphere was discussed in \cite{abouelaz2001} and the Grassmanian was done in \cite{camporesi2006}. But so far the case of the full space $C_r^\infty (U/K)$ is still open.  Let us describe the main results in \cite{OlafssonSchlichtkrull2010}. For simplicity we will always assume that $U/K$ is simply connected and note that the results in \cite{OlafssonSchlichtkrull2010} are more general than stated here.

Compact and noncompact symmetric spaces come (up to a covering) in pairs. To use the notation that we already introduced, let $\fq=i\fs$, $\fb=i\fa$ and $\fu=\fk\oplus \fq$. Let $\fg_\C=\fg\otimes_\R \C$ be the complexification of $\fg$ and let $G_\C$ denote a simply connected Lie group with Lie algebra $\fg_\C$. Let $U$ be the subgroup of $G_\C$ with Lie algebra $\fu$. Then $U$ is compact and simply connected. We will assume that $G\subset G_\C$. The involution $\theta$ extends to an involution on $G_\C$. By restriction it defines an involution on $U$ as well, which we also denote by $\theta$. $U^\theta$ is connected and $U^\theta=U\cap G=K$. Let $Y:= U/K$ and $X_\C= G_\C/K_\C$, where from now on the subscript ${}_\C$ denotes the complexification in $G_\C$ of real subgroups in $G$ or $U$.

Let
\begin{equation}\label{Lambda}
\Lambda^+ (Y):=\left\{ \mu \in \fa^*\: (\forall \alpha\in\Sigma^+)\,\, \frac{(\mu ,\alpha )}{(\alpha ,\alpha)}\in \N \right\}
\end{equation}
and denote by $\widehat{U}_K$ the set of equivalence classes of irreducible representations with a non-trivial $K$-fixed vector. If $\pi$ is an irreducible representation of $U$, then $[\pi ]$ denotes the equivalence class of $\pi$.

\begin{thm} If $\mu \in\Lambda^+(Y)$, then there exists a unique irreducible representation $\pi_\mu$ with highest weight $\mu$. $[\pi_\mu]\in \widehat{U}_K$ and the map $\Lambda^+(Y)\ni \mu \mapsto [\pi_\mu]\in  \widehat{U}_K$ is a bijection. Furthermore, if $[\pi ]\in  \widehat{U}_K$, then $\dim V_\pi^K=1$, where $V_\pi$ denotes the Hilbert space on which $\pi$ acts.
\end{thm}

\begin{proof} See \cite{helgason2008}, p. 538.
\end{proof}

For $\mu\in \Lambda^+ (Y)$, let $(\pi_\mu,V_\mu)$ be an irreducible representation with the highest weight $\mu$. Let $d (\mu)=\dim_\C V_\mu$. Let $e_\mu$ be a $K$-fixed vector of norm one and let $v_\mu$ be a highest weight vector such that $(e_\mu,v_\mu)=1$. Recall that $\pi_\mu$ extends to a holomorphic representation of $G_\C$. We have with $M_\C = Z_{K_\C}(A_\C)$
\begin{equation}\label{eq-piMvmu}
\pi_\mu (m)v_\mu=v_\mu\quad \text{ for all } m\in M_\C\, ,
\end{equation}
see \cite{helgason2008}, p. 535 and \cite{OlafssonSchlichtkrull2010}, Lemma 3.1. Note, in \cite{OlafssonSchlichtkrull2010} this was proved for $M$ only, but the claim follows from that, because $M_\C=(M_\C)_oM$ and $\pi_\mu|_{M_\C}$ is holomorphic.

Let $f\in C_c^\infty (Y)$, $k\in K$, and $\mu\in\Lambda^+(Y)$. Using that
\begin{equation}\label{eq-piMuIntK}
\pi_\mu (u)e_\mu = \int_B e^{-(\mu +2\rho)(H(u^{-1}k))}\pi_\mu (k)v_\mu \, dk=
\int_B e_{\mu+\rho, kM}(u)\pi_\mu (k)v_\mu \, dk
\end{equation}
(see the proof of Lemma 3.2 in \cite{OlafssonSchlichtkrull2010}) we have
\begin{eqnarray*}
f(x)&=&\sum_{\mu\in\Lambda^+(Y)} d(\mu )(\widehat{f}_\mu ,\pi_\mu (x)e_\mu)\\
&=& \sum_\mu d (\mu )\int_U f(u)(\pi_\mu (u)e_\mu , \pi_\mu (x)e_\mu)\\
&=&\sum_\mu d(\mu )\int_Y \int_U f(u) e_{\mu +\rho,b}(u)(\pi_\mu (x^{-1}k)v_\mu ,e_\mu )\, dk\\
&=& \sum_\mu d(\mu )\int_Y \left(\int_U f(u)e_{\mu + \rho, b} (u)\, du\right)\,
e_{-\mu -\rho, b}(x)\, db\, .
\end{eqnarray*}
Therefore, following T. Sherman \cite{sherman75,sherman77,sherman90}, we define
\begin{equation}\label{def-SherFT}
\widetilde{f}(\mu ,b):=\int_Y f(u)e_{\mu + \rho, b}(u)\, du\, .
\end{equation}
Thus $\widetilde{f} : \Lambda^+(Y)\to C^\infty (B)$.
Note that this definition differs from the one in \cite{OlafssonSchlichtkrull2010} by a $\rho$-shift and a minus-sign.

To clarify this and to determine for which spaces of functions this is well defined, we recall that $K_\C A_\C N_\C$ and $N_\C A_\C K_\C$ are open complex submanifolds in $G_\C$. But the $A_\C$ component is not uniquely determined anymore because $\{e\}\subsetneqq A_\C\cap K_\C\not\subset M_\C$. However, by (\ref{eq-piMvmu}) and $2\rho \in \Lambda^+(Y)$, it follows that $e_{\mu +\rho,b}(x)$ is well defined for $x^{-1}k\in K_\C A_\C N_\C$ and $\mu\in \Lambda^+(Y)$. For a Paley-Wiener type theorem we need a set where the holomorphic extension in $\lambda$ is well defined on all of $\fa_\C^*$. For that one shows that there exists a $K$-invariant domain $\cU_1\subset X_\C$ containing $X$ such that $(xK,kM)\mapsto a(x^{-1}k)\in A_\C$ is well defined and that there exists a $K$-invariant subset $\cU \subset X_\C$, containing $X$  such that $(xK,b)\mapsto e_{\lambda ,b}(x)$ is well defined and holomorphic as a function of $xK$ and $\lambda$. We refer to the discussion and references in \cite{OlafssonSchlichtkrull2010}. The holomorphic continuation of $\widehat{f}$, denoted by $\widehat{f}^c$, is well defined if $f\in C_c^\infty( \cU\cap Y)$. Furthermore, we have an intertwining relation
\[ \cA (w, -\mu-\rho)\widetilde{f}(\mu)=\widetilde{f}(w(\mu+\rho)-\rho)\]
or equivalently
\[\cA (w, -\mu )\widetilde {f}(\mu -\rho )= \widetilde f (w\mu -\rho )\, .\]

Let $R>0$ be smaller than the injectivity radius for $Y$ and so that every closed ball in $Y$ of radius $0<r\le R$ is contained in $\Xi \cap Y$. Denote by $C^\infty_{F,r} (Y)$ the space of $K$-finite functions on $Y$ with support in a closed ball of radius $r$, and similarly $C^\infty_F(B)$ the space of $K$-finite functions on $B$. For $r<R$ let $\PW_r(\fb_\C^*,C^\infty_F (B))$ denote the space of holomorphic functions $\varphi $ on $\fb_\C^*$ such that
\begin{enumerate}
\item  $\varphi(\mu ,\, \cdot\, )\in C^\infty_F (B)$ the $K$-types are independent of $\mu$.
\item For all $N\in \N$, $\sup_{\lambda\in\fb^*_\C} (1+|\lambda |^2)^N e^{-r|\Re (\lambda )|} \|\varphi (\lambda )\| < \infty$.
\item We have for all $w\in W$ and $\lambda\in\fb_\C^*$ that $\cA (w,-\lambda )\varphi (\lambda - \rho)=\varphi (w\lambda -\rho)$.
\end{enumerate}

Note that by (1) there exists a finite dimensional $K$-invariant subspace $\cH_\varphi\subset C^\infty (B)$ such that $\varphi : \fb_\C^*\to \cH_\varphi$. Therefore in (2) one can use other topologies on $C^\infty_F(B)$, like the supremum of derivatives or the weak topology of $\cH_\varphi$ as a subspace of $L^2(B)$.

\begin{thm}[\cite{OlafssonSchlichtkrull2010}] Suppose that $0<r<R$.
 \begin{enumerate}
\item Let $f\in C^\infty_{F,r}(Y)$. Then $\mu \mapsto \widetilde{f}(\mu)$ extends to a holomorphic function $\widetilde{f}^c$ on $\fb^*_\C$ and $\widetilde{f}^c\in  \PW_r(\fb_\C^*,C^\infty_F (B))$.
\item If $\varphi \in  \PW_r(\fb_\C^*,C^\infty_F (B))$, then there exists $f\in
C^\infty_{F}(Y)\cap C^\infty_{F,r}(Y)$ such that $\widetilde{f}(\mu)=\varphi (\mu)$ for all $\mu\in \Lambda^+ (Y)$.
\item There exists $0< S\le R$ such that for all $0<r\le S$ the map $C_{F,r}^\infty (Y)\to \PW_r(\fb_\C^*,C^\infty_F (B))$ is a linear isomorphism.
\end{enumerate}
\end{thm}
\begin{proof} See \cite{OlafssonSchlichtkrull2010}.
\end{proof}

Let us make a few comments on this theorem, its proof, and the different $R$ and $S$ that show up in the statement of the theorem. The proof is by reduction to the $K$-invariant case (\cite{OlafssonSchlichtkrullPW08}) using Kostant's description of the spherical principal series \cite{kostant75}, see also \cite[Ch. III]{helgason2008}.   This is an idea that was already used by P. Torasso in \cite{torasso77}. We would like to point out, that in \cite{torasso77} the fact that the Helgason Fourier transform on $X$ maps $C^\infty_r(X)$ into $\Hr$ (without the $K$-finiteness condition) was proved using the Fourier-Slice theorem (\ref{eq-FSl2}).

For the $K$-invariant case, as mentioned earlier, $\widehat{f}_\mu$ is a multiple of $e_\mu$, $\widehat{f}_\mu =(\widehat{f}_\mu, e_\mu)e_\mu$. We have
\[(\widehat{f}_\mu ,e_\mu)=\int_U f(u\cdot x_o)(\pi_\mu (u)e_\mu , e_\mu)\, du=\int_Y f(y)\psi_\mu (y)\, dy\]
where $\psi_\mu$ is the \textit{spherical function} $u\mapsto (\pi_\mu (u)e_\mu ,e_\mu)$. As $\pi_\mu$ extends to a holomorphic representation of $G_\C$, it follows that $\psi_\mu$ extends to a $K_\C$-invariant holomorphic function on $X_\C$. By (\ref{eq-sphericalFunction}) and (\ref{eq-piMuIntK}) we get that
\begin{equation}\label{eq-spherFunctEq}
\psi_\mu|_X =\varphi_{\mu+\rho}\, .
\end{equation}

According to \cite{BransonOlafssonPasquale2005,KrotzStanton2005} the function $x\mapsto \varphi_\mu$ extends to a $K_\C$-invariant holomorphic function on $\widetilde{\Xi}=K_\C\exp (2i\Omega)\cdot x_o$ where $\Omega:=\{X\in \fa\: (\forall \alpha \in \Sigma)\,\, |\alpha (X)|<\pi /2\}$. This gives a holomorphic extension $\widehat{f}^c$ of $\mu \mapsto \widehat{f}(\mu )$ for $f$ $K$-invariant and supp$(f)\subseteq \overline{\widetilde{\Xi}\cap Y}$:
\[\widehat{f}^c(\lambda )=\int_U f(x)\varphi_{\lambda +\rho}(x)\, dx\, .\]
The holomorphic extension satisfies $\widehat{f}^c(\lambda) =\widehat{f}^c(w(\lambda + \rho)-\rho)$ because of the Weyl group invariance of $\lambda \mapsto \varphi_\lambda$.

To show that $\widetilde{f}^c$ has exponential growth one needs to show that the spherical functions are of exponential growth. That has been shown only on $\Xi = K_\C \exp (\Omega )\cdot x_o$, see \cite{opdam95}, Theorem 6.1. Thus $R$ has to be so that $\overline{B_R(x_o)} \subseteq \Xi\cap Y$, forcing $R$ to be, in general, much smaller than the injectivity radius.

For $Y=\bs^n =\SO (n+1)/\SO (n ) =\SO (n+1)\cdot e_1$ we have $\Sigma=\{\alpha ,-\alpha\}$ such that $\alpha (H_o) =1$ with $H_o=E_{2,1}-E_{1,2}$, $E_{\nu,\mu}=(\delta_{i\nu}\delta_{j\mu})_{i,j}$ and $\Lambda^+(\bs^n)=\N \alpha$. Therefore we view the spherical Fourier transform of $f$ as a function $\widehat{f} : \N\to \C$. Note that $\exp (tH_o)=\cos (t)e_1+\sin (t)e_2$. Thus $\widetilde{\Xi}\cap Y=\bs^n \setminus \{-e_1\}$.
For the holomorphic extension of $\widetilde{f}$, we only need $f$ to vanish at $\{-e_1\}$. But
\begin{eqnarray*}
\Xi\cap \bs^n &=&\SO (n)\cdot \{ \cos (t)e_1+\sin (t)e_2 \: -\frac{\pi}{2}<t< \frac{\pi}{2}\}\\
							&=&\{(x_1,\ldots ,x_{n+1})\: x_1\ge 1\}.
\end{eqnarray*}
Thus, for the exponential growth  we have to assume that the support of $f$ is contained in the upper hemisphere $\bs^n_+=\{(x_1,\ldots ,x_{n+1})\: x_1\ge 1\}$.

The constant $S$ is needed because of the Carlson's theorem, \cite[p. 153]{boas54} and \cite[Lemma 7.1]{OlafssonSchlichtkrullPW08}: Let $F :\C^n\to \C$ be a holomorphic function such that $F(z)=0$ for all $z\in \N^n$, $j=1,\ldots ,n$. We have $|F(z+\eta e_j)|\le C_1 e^{\tau |\eta|}$ and $|F(z+iy e_j)|\le C_2 e^{c|y|}$ for some constants $C_1,C_2,\tau$ and $0\le c<\pi$, and so $F=0$. Taking $F(z)=\sin (\pi z)$ shows that the condition $c<\pi$ is necessary.

In \cite{abouelaz2001} a Paley-Wiener theorem for the sphere was proved for $K$-invariant functions $f$ such that
the function $t\mapsto f(\cos (t)e_1+\sin (t)e_2)$ and its first $n-3$ derivatives vanishes at $t=\pi$. The main idea of the proof is a Fourier-Slice type theorem. Identify $K$-invariant functions on the sphere with even functions on $[-\pi,\pi]$ by $F(t)=f(\cos (t)e_1+\sin (t)e_2)$. Note that $\rho=(n-1)/2$ if we identify $\fa^*_\C$  with the complex plane by $z\mapsto z\alpha$. For $0\le s\le \pi$ define
\[\cR (f)(s ):=\frac{2^\rho\rho}{\pi}\int_{s}^\pi F(t)\sin (t)(\cos (s)-\cos (t))^{\rho-1}\, dt\, .\]
Then we have the Fourier-Slice theorem \cite[Thm. 6]{abouelaz2001}:
\begin{equation}\label{eq-FSlSn}
\widehat{f}(m )=c\int_0^\pi \cos ((m+\rho)t)\cR(f) (t)\, dt\, .
\end{equation}

Now the Theorem 9 in \cite{abouelaz2001} says that if $f\in C^\infty (Y)^K$ vanishes of order $n-3$ at the south pole, then for $0< r \le \pi$, supp$(\cR(f))\subseteq [-r,r]$ if and only if supp$(F)\subseteq [-r,r]$. This, together with the Fourier-Slice theorem (\ref{eq-FSlSn}), implies that $\widetilde{f}$ extends to a holomorphic function of exponential growth such that $\widetilde{f}^c (-z-\rho )=\widetilde{f}^c(z-\rho)$. The fact that every holomorphic function of exponential growth $r$ is a holomorphic extension of a smooth function with support in $K\exp ([0,r])\cdot e_1$ is proved in a similar way using the inversion formula for the Radon transform. So, again, the Fourier-Slice theorem plays a fundamental role! We note, that if $n$ is odd, then $m_\alpha =n-1$ is even. Hence that case is also covered by \cite[Thm. 38]{BransonOlafssonPasqualeH}. It would be interesting to generalize the approach in \cite{abouelaz2001} to other rank one spaces or even higher rank compact symmetric spaces.

Staying with the example $Y=\bs^n$ we note that now
\[B=\SO (n)/\SO (n-1)=\{(0,x)\: x\in \sn\}\, .\]
The ``exponential function'' $e_{m,b}(x)$ is given by
\[e_{m,b}(z)=(z,(1,ib))^m=(z_1+i(z_2b_2+\ldots +z_{n+1}b_{n+1}))^m\, .\]
Hence $e_{\lambda ,b}(z)$ is well defined for all $\lambda$ as long as $z$ is in the domain  $\{z\in \bS^n_\C\: (\forall b\in \sn )\,\, z_1+i(z_2b_2+\ldots +z_{n+1}b_{n+1})\in \C\setminus (-\infty ,0]\}$. If $z\in \bs^n$ and $b\in \sn$, then that is equivalent to $z_1>0$, i.e., $z\in \bS^n_+$.

\section{The Inductive Limit of Symmetric Spaces}\label{SectionILofSS}
\noindent
One of the interesting aspects of the Paley-Wiener theory for $\R^n$ and semisimple symmetric spaces is that many of these results extend to some special classes of inductive limits of these spaces, see \cite{OlafssonWolf2009}.

The Euclidean case is a consequence of the results by Cowling \cite{cowling86} and Rais \cite{rais83}. Let $k\ge n$ and view $\R^n\simeq$ as a subspace of $\R^k$ by
 \[\R^n\simeq \{(x_1,\ldots ,x_n, 0,\ldots ,0)\: x_j\in \R\}\subseteq \R^k\, .\]
Assume that $W (n)$ is a finite reflection group acting on $\R^n$ and that $W(k)$ is a finite reflection group acting on $\R^k$. Set
\begin{equation}\label{defWnk}
W_n(k):= \{w\in W(k)\: w(\R^n)=\R^n\}\, .
\end{equation}
Then $W_n(k)$ is a subgroup of $W(k)$. Denote by $\C[\R^n]$ the algebra of polynomial maps $\R^n\to \C$. A subgroup $G\subseteq \GL (n,\R)$  acts on $\C [\R^n]$ by $g\cdot p(x)=p(g^{-1}(x))$. We denote by $\C [\R^n]^{G}$ the algebra of invariant polynomials.

\begin{thm}[\cite{OlafssonWolf2009}, Theorem 1.9]\label{OW1} Assume that $W_{n}(k)|_{\R^n}=W(n)$ and that the restriction map
$ \C [\R^k]^{W(k)}\to \C [\R^n]^{W(n)}$ is surjective. Then the restriction map
\[R^k_n:\PW_r(\C^k)^{W(k)}\to \PW_r(\C^n)^{W(n)}\, ,\quad F\mapsto F|_{\C^n}\]
is surjective for all $r>0$.
\end{thm}

\begin{proof} It is clear that   $R^k_n (\PW(\C^k)^{W(k)})\subseteq \PW(\C^n)^{W(n)} $. For the surjectivity let $G\in \PW(\C^n)^{W(n)} $. By the surjectivity result in \cite{cowling86} there exists $G\in \PW_r(\C^k)$ such that $G|_{\C^n}=F$. As $F$ is $W(n)$-invariant and by our assumption that
$W_{n}(k)|_{\R^n}=W(n)$, we can average $G$ over $W_n(k)$. Hence we can assume that $G$ is $W_n(k)$ invariant.  According to \cite{rais83}, there exists $G_1,\ldots ,G_k\in  P(\R^k)^{W(k)}$ and $p_1,\ldots ,p_k\in P(\R^k)$ such that
\[G=p_1G_1+\ldots +p_kG_k\, .\]
Again, by averaging, we can assume that $p_j\in P(\R^k)^{W_n(k)}$. But then $p_j|_{\R^n}\in P(\R^n)^{W(n)}$ and by our assumption there exists $q_j\in P(\R^k)^{W(k)}$ such that $q_j|_{\R^n} =p_j$. Let $H:=q_1G_1+\ldots +q_kG_k$. Then $H\in \PW_r(\C^k)^{W(k)}$ and $H|_{\C^n}=F$.
\end{proof}

Note that the on the level of smooth functions, the above restriction map corresponds to
\begin{equation}\label{eq-contProj}
f\mapsto (x\mapsto C^k_n(f)(x):=\int_{(\R^n)^\perp} f(x,y)\, dy )
\end{equation}
which by Theorem \ref{OW1} induces a surjective map $C^\infty_r(\R^k)\to C^\infty_r(\R^n)$.

Theorem \ref{OW1} leads to a projective system $\{(P(\R^k)^{W(k)},R^k_n)\}$ with surjective projections. The projective limit $\varprojlim \PW_r(\C^n)^{W(n)}$ with the surjective projection $R^\infty_k :\varprojlim \PW_r(\C^n)^{W(n)}\to  \PW_r(\C^k)^{W(k)}$  can be viewed as a space of functions on $\R_\infty =\varinjlim \R^n$, the space of all finite real sequences, by
\[F ((x_1,\ldots ,x_k,0,\ldots ))=  R^\infty_k(F)(x_1,\ldots ,x_k)\, .\]
It is easy to see that this definition is independent of the choice of $k$ such that $(x_j)\in \R^k$. Furthermore, $R^\infty_k$ is surjective. Similar statement holds for compactly supported smooth functions by using the projection maps (\ref{eq-contProj}), see the commutative diagram (\ref{eq-commutativeDia}) which can also be used for $\R^n$.

Without going into details, we note that we can also have projective limits for the space $\cD_{H,R}(\Xi_n)$ and $\widetilde{\PW}^{\Z_2}_{H,R}(\C\times \bS_{\C}^{n-1})$ and that the commutative diagram (\ref{eq-ComDia1}) gives a similar commutative diagram for the limits. In this setting the limit of the vertical arrows has a nice interpretation as an infinite dimensional Radon transform. First embed $\bS^{n-1}$ into $\bS^{k-1}$. Then we have a well defined restriction map $r^k_nf:= f|_{\mathbb{R}\times \bS^{n-1}}$ where $f$ is a function on $\R\times \bS^{k-1}$ by $\omega \mapsto (\omega ,0)$. If $f\in C^\infty_r(\Rn )$ then (\ref{eq-contProj}) implies that
\begin{equation}\label{eq-rad}
\cR_{\R^n}(C^k_n(f))(p,\omega) =r^k_n(\cR_{\R^k}(f))(p,\omega )\, ,\quad p\in\R ,\, \omega \in \bS^{n-1}
\end{equation}
which leads to surjective map
\[\cR_\infty : \varprojlim C^\infty_R(\R^n)\to \varprojlim C^\infty_{H,R}(\Xi_n)\]
such that
\[r^\infty_n (\cR_\infty (F))=\cR_n(C^\infty_n(F))\, .\]

Let us now turn our attention to symmetric spaces. To avoid introducing too much new notation we will concentrate on symmetric spaces of noncompact type and only say a few words about the compact case. We use the notation from pervious sections and add to it an index $n$ or $(n)$ wherever needed, to indicate the dependence of the symmetric space $X_n=G_n/K_n$ or $Y_n=U_n/K_n$ on $n$.
Let $\Sigma_{1/2} :=\{\alpha\in\Sigma \: \frac{1}{2}\alpha\not\in \Sigma\}$. Then $\Sigma_{1/2}$  is a root system. From now on we assume that $\Sigma_{1/2}$ is classical as finitely many exceptional cases can be removed from any projective sequence without changing the limit. Let $\Psi:=\{\alpha_1,\, \ldots ,\alpha_k\}$ be the set of simple roots. We number the roots so that the corresponding Dynkin diagram is
\begin{equation}\label{rootorder}
\begin{aligned}
&\begin{tabular}{|c|l|c|}\hline
$\Psi=A_k$&
\setlength{\unitlength}{.5 mm}
\begin{picture}(155,18)
\put(5,2){\circle{2}}
\put(2,5){$\alpha_{k}$}
\put(6,2){\line(1,0){13}}
\put(24,2){\circle*{1}}
\put(27,2){\circle*{1}}
\put(30,2){\circle*{1}}
\put(34,2){\line(1,0){13}}
\put(48,2){\circle{2}}
\put(49,2){\line(1,0){23}}
\put(73,2){\circle{2}}
\put(74,2){\line(1,0){23}}
\put(98,2){\circle{2}}
\put(99,2){\line(1,0){13}}
\put(117,2){\circle*{1}}
\put(120,2){\circle*{1}}
\put(123,2){\circle*{1}}
\put(129,2){\line(1,0){13}}
\put(143,2){\circle{2}}
\put(140,5){$\alpha_1$}
\end{picture}
&$k\geqq 1$
\\
\hline
\end{tabular}\\
&\begin{tabular}{|c|l|c|}\hline
$\Psi=B_k$&
\setlength{\unitlength}{.5 mm}
\begin{picture}(155,18)
\put(5,2){\circle{2}}
\put(2,5){$\alpha_{k}$}
\put(6,2){\line(1,0){13}}
\put(24,2){\circle*{1}}
\put(27,2){\circle*{1}}
\put(30,2){\circle*{1}}
\put(34,2){\line(1,0){13}}
\put(48,2){\circle{2}}
\put(49,2){\line(1,0){23}}
\put(73,2){\circle{2}}
\put(74,2){\line(1,0){13}}
\put(93,2){\circle*{1}}
\put(96,2){\circle*{1}}
\put(99,2){\circle*{1}}
\put(104,2){\line(1,0){13}}
\put(118,2){\circle{2}}
\put(115,5){$\alpha_2$}
\put(119,2.5){\line(1,0){23}}
\put(119,1.5){\line(1,0){23}}
\put(143,2){\circle*{2}}
\put(140,5){$\alpha_1$}
\end{picture}
&$k\geqq 2$\\
\hline
\end{tabular} \\
&\begin{tabular}{|c|l|c|}\hline
$\Psi=C_k$ &
\setlength{\unitlength}{.5 mm}
\begin{picture}(155,18)
\put(5,2){\circle*{2}}
\put(2,5){$\alpha_{k}$}
\put(6,2){\line(1,0){13}}
\put(24,2){\circle*{1}}
\put(27,2){\circle*{1}}
\put(30,2){\circle*{1}}
\put(34,2){\line(1,0){13}}
\put(48,2){\circle*{2}}
\put(49,2){\line(1,0){23}}
\put(73,2){\circle*{2}}
\put(74,2){\line(1,0){13}}
\put(93,2){\circle*{1}}
\put(96,2){\circle*{1}}
\put(99,2){\circle*{1}}
\put(104,2){\line(1,0){13}}
\put(118,2){\circle*{2}}
\put(115,5){$\alpha_2$}
\put(119,2.5){\line(1,0){23}}
\put(119,1.5){\line(1,0){23}}
\put(143,2){\circle{2}}
\put(140,5){$\alpha_1$}
\end{picture}
& $k\geqq 3$
\\
\hline
\end{tabular}\\
&\begin{tabular}{|c|l|c|}\hline
$\Psi=D_k$ &
\setlength{\unitlength}{.5 mm}
\begin{picture}(155,20)
\put(5,9){\circle{2}}
\put(2,12){$\alpha_{k}$}
\put(6,9){\line(1,0){13}}
\put(24,9){\circle*{1}}
\put(27,9){\circle*{1}}
\put(30,9){\circle*{1}}
\put(34,9){\line(1,0){13}}
\put(48,9){\circle{2}}
\put(49,9){\line(1,0){23}}
\put(73,9){\circle{2}}
\put(74,9){\line(1,0){13}}
\put(93,9){\circle*{1}}
\put(96,9){\circle*{1}}
\put(99,9){\circle*{1}}
\put(104,9){\line(1,0){13}}
\put(118,9){\circle{2}}
\put(113,12){$\alpha_3$}
\put(119,8.5){\line(2,-1){13}}
\put(133,2){\circle{2}}
\put(136,0){$\alpha_1$}
\put(119,9.5){\line(2,1){13}}
\put(133,16){\circle{2}}
\put(136,14){$\alpha_2$}
\end{picture}
& $k\geqq 4$
\\
\hline
\end{tabular}
\end{aligned}
\end{equation}
We note that $\Sigma_{1/2}=\Sigma$ except in the cases $\SU (p,q)/\SU (p+q)$, $\Sp (p,q)/\Sp (p)\times \Sp (p)\times \Sp (q)$ for  $1\le p <q$, $\SO^*(2j)$ for $j$ odd, and for the compact dual spaces.

Let $X_1=G_1/K_1\subseteq X_2=G_2/K_2$ be two irreducible symmetric spaces of the compact or noncompact type. We say that $G_2/K_2$ propagates $G_1/K_1$ if the following holds (with the obvious notation):
\begin{enumerate}
\item $G_1\subseteq G_2$, $K_1\subseteq K_2$, at least up to covering, and hence $G_1/K_1\hookrightarrow G_2/K_2$ and    $\fs_1\subseteq \fs_2$,
\item If we choose $\fa_1\subseteq \fa_2$, then $\Sigma (1)\subseteq \{\alpha_{\fa_1}\: \alpha\in\Sigma (2)\}$ and
the Dynkin diagram for $\Psi (2)$ is gotten from that of $\Psi (1)$ by adding simple roots at the \textit{right} end of the Dynkin diagram for $\Psi (1)$.
\end{enumerate}
Simple examples are $\bs^n\subseteq \bs^k$ and $G_{i,n} (\K)\subseteq G_{i,k} (\K)$ for $k\ge n$, where $G_{i,j} (\K)$ stands for the space of $i$-dimensional subspaces of $\K^j$, and $\K=\R$, $\C$, or $\H$.

In general we say that the symmetric space $G_2/K_2$ propagates $G_1/K_1$ if we can write $G_2/K_2$ up to covering as $G_2^1/K_2^1\times \ldots \times G_2^n/K_2^n$ where each $G_2^j/K_2^j$ is irreducible and similarly $G_1/K_1$ locally isomorphic to $G_1^1/K_1^1\times \ldots \times G_1^k/K_1^k$ with $k\le n $ such that $G_2^j/K_2^j$ is a propagation of $G_1^j/K_1^j$ for $j\le k$. From now on we will assume that $G_2/K_2$ is a propagation of $G_1/K_1$ and that $G_j/K_j$, $j=1,2$, is of noncompact type. Similarly $U_2/K_2$ is a propagation of $U_1/K_1$ and $U_j/K_j$, $j=1,2$ is of compact type.  We will always assume that $\fa_1\subseteq \fa_2$.

\begin{thm}\label{th-AdmExtG/K} Assume that $X_k$ and $X_n$ are
symmetric spaces of compact or noncompact type and that $X_k$ propagates
$X_n$.  Denote by $W (n)$, respectively $W (k)$, the Weyl group related to $X_n$, respectively $X_k$. Let $W_n(k):=\{w\in W(k)\: w(\fa_n) = \fa_n\}$.

\begin{enumerate}
\item  If $X_n$ does not contain any irreducible factors
with $\Psi_{1/2} (n)$ of type $D$, then
\begin{equation}\label{eq-RestrictionOfWeyl1}
W_{n}(k)|_{\fa_n} = W(n)
\end{equation}
and the restriction map $\C [\fa_k]^{W(k)}\to \C [\fa_n]^{W(n)}$ is surjective.
\item Assume that $X_n$ and $X_k$ are of type $D$. Then
$W (n)$ is the group $\gamma_n$ of permutations of $n$ objects semidirect product with all even number of sign changes whereas $W_n(k)|_{\fa_n}$ is the group $\gamma_n$ semidirect product of all sign changes.
 \item If $X_n$ and $X_k$ are of type $D$, then $\C [\fa_k]^{W(k)}|_{\fa_n}$ is the algebra of even $\gamma_n$-invariant polynomials and $\C [\fa_k]^{W(k)}|_{\fa_n}\subsetneqq P (\fa_n)^{W(n)}$.
\end{enumerate}
\end{thm}

We remark the Pfaffian in $\C [\fa_k]^{W(k)}$ restricts to zero and all the elements in $\C [\fa_k]^{W(k)}|_{\fa_n}$ are even in the $D$-case. That is why the restriction map in (3) fails to be surjective.

We will from now on assume that $X_j$ (or  $Y_j$) is a sequence of symmetric spaces so that $X_k$ propagates $X_n$ for $n\le k$. We will also assume that none of the spaces contains a factor of type $D$ and similarly for the compact spaces $Y_j$. For the general statement we refer to \cite{OlafssonWolf2009}. Denote by $\PW_{r}(\fa_{j\C}^*)^{W(j)} $ the space of $K_j$-invariant elements in $\PW_{r,W (j)}(\fa_{j\C}^*,C^\infty (B_j))$.

\begin{thm}[\cite{OlafssonWolf2009}] Let $\{X_n\}$ be as above. Then the restriction maps
\[R^k_n :\PW_{r}(\fa_{k\C}^*)^{W(k)} \to \PW_{r}(\fa_{n\C}^*)^{W(n)} \]
and
\[R^\infty :\varprojlim \PW_{r}(\fa_{n\C}^*)^{W(n)} \to \PW_{r}(\fa_{k\C}^*)^{W(k)} \]
are surjective. In particular,
$\varprojlim \PW_{r,W(n)}(\fa_{n\C}^*)\not=\{0\}$.
\end{thm}

We remark, that the corresponding projection $C^\infty_r(X_k)\to C^\infty_r(X_n)$ is more complicated than (\ref{eq-contProj}) because of the Harish-Chandra $c$-function and the fact that the spherical functions on $X_k$ do not necessarily project into spherical functions on $X_n$. But the projection are still given by the sequence
\[\text{Spherical Fourier transform on } X_k \rightarrow  \text{ restriction } R^k_n\]
\[{}\hbox to2.3cm{} \rightarrow  \text{ inverse spherical Fourier transform on } X_n\, .
\]
We denote this map by $S^k_n$.
Thus we have a commutative diagram
\begin{equation}\label{eq-commutativeDia}
{
\xymatrix{C^\infty_r(X_n)^{K_n}\ar[d]_{\cF_n}&
C^\infty_r(X_{n+1})^{K_{n+1}}
\ar[d]_{\cF_{n+1}}\ar[l]_(0.55){S^{n+1}_n}&\ar[l]_(0.4){S^{n+2}_{n+1}}
\cdots&
\varprojlim C^\infty_r(X_{n})^{K_{n}}\ar[d]_{\cF_\infty}
\\
\PW_{r}(\fa^*_{n\C})^{W(n)} &  \PW_{r}(\fa^*_{n+1\, \C})^{W(n+1)}
\ar[l]^(0.55){R^{n+1}_{n}}
&\ar[l]^(0.35){R^{n+2}_{n+1}}  \cdots & \varprojlim
\PW_{r}(\fa^*_{n\C} )^{W(n)} \\
}
}
\end{equation}
which we can interpret as an infinite dimensional Paley-Wiener type theorem for $X_\infty
=\varinjlim X_n$. The vertical maps are isomorphisms.

Similar results can also be derived for the compact case. To avoid the introduction of additional notation as well as needed preliminaries related to the representation theory of $U_n$ and $U_k$, we refer to \cite{OlafssonWolf2009} for the details. We only point out, that on the level on the Paley-Wiener spaces, we need to use the $\rho_n$-translated space $\{L_{\rho_n}F= F(\cdot - \rho_n)\:
F\in \PW_r (\fb^*_n,C^\infty (B))\}$ which is nothing elso but $\PW_{r,W (n)}(\fa^*_{n\C})$. Similarly, we will need the $\rho_n$-shifted Fourier transform, $f\mapsto L_{\rho_n}\widetilde{f}$. We can think of $\fg_n$ as a Lie algebra of matrices and use $(X,Y)_n=\Tr (XY)$ as a $K$-invariant inner product on $\fs_n$. Then $(X,Y)_k= (X,Y)_n$ if $Y_k$ propagates $Y_n$ and the injectivity radius stays constant. However, we have to replace $\Omega_n$ by a smaller convex set $\Omega^*_n$  such that for $k\ge n$.  We have $\Omega^*_k\cap \fa_n= \Omega^*_n$. An explicit definition of $\Omega^*_n$ is given in \cite{OlafssonWolf2009}. With those adjustments the commutative diagram (\ref{eq-commutativeDia}) stays valid for $r$ small enough. Here the vertical lines correspond to the holomorphic extension given by the spherical functions followed by a $\rho_n$-shift. Those maps are not necessarily isomorphisms any more because of the Carlson's theorem. In fact, the constant needed for that might tend to zero as $n\to \infty$.

\nocite{}
\bibliographystyle{amsplain}
\bibliography{ref}

\end{document}